\documentclass[11pt]{amsart} 
\usepackage{amssymb}
\usepackage{verbatim}
\usepackage{hyperref}
\usepackage{color}

\usepackage{tikz}

\begin{document}
\setcounter{tocdepth}{1}

\newtheorem{theorem}{Theorem}    
\newtheorem{proposition}[theorem]{Proposition}
\newtheorem{conjecture}{Conjecture}
\newtheorem{corollary}[theorem]{Corollary}
\newtheorem{lemma}[theorem]{Lemma}
\newtheorem{sublemma}[theorem]{Sublemma}
\newtheorem{fact}[theorem]{Fact}
\newtheorem{observation}[theorem]{Observation}
\newtheorem{definition}{Definition}
\newtheorem{notation}[definition]{Notation}
\newtheorem{remark}[definition]{Remark}
\newtheorem{question}[conjecture]{Question}
\newtheorem{questions}[conjecture]{Questions}

\newtheorem{example}[definition]{Example}
\newtheorem{problem}[definition]{Problem}
\newtheorem{exercise}[definition]{Exercise}

 \numberwithin{theorem}{section}
 \numberwithin{definition}{section}
 \numberwithin{equation}{section}

\newcommand{\cA}{\mathcal{A}}
\newcommand{\cB}{\mathcal{B}}
\newcommand{\cC}{\mathcal{C}}
\newcommand{\cD}{\mathcal{D}}
\newcommand{\cE}{\mathcal{E}}
\newcommand{\cG}{\mathcal{G}}
\newcommand{\cH}{\mathcal{H}}
\newcommand{\cI}{\mathcal{I}}
\newcommand{\cJ}{\mathcal{J}}
\newcommand{\cK}{\mathcal{K}}
\newcommand{\cL}{\mathcal{L}}
\newcommand{\cM}{\mathcal{M}}
\newcommand{\cN}{\mathcal{N}}
\newcommand{\cO}{\mathcal{O}}
\newcommand{\cS}{\mathcal{S}}
\newcommand{\cT}{\mathcal{T}}
\newcommand{\cU}{\mathcal{U}}
\newcommand{\cV}{\mathcal{V}}
\newcommand{\cW}{\mathcal{W}}
\newcommand{\cX}{\mathcal{X}}
\newcommand{\cY}{\mathcal{Y}}
\newcommand{\cZ}{\mathcal{Z}}
\newcommand{\bA}{\mathbb{A}}
\newcommand{\bB}{\mathbb{B}}
\newcommand{\bC}{\mathbb{C}}
\newcommand{\bD}{\mathbb{D}}
\newcommand{\bF}{\mathbb{F}}
\newcommand{\bG}{\mathbb{G}}
\newcommand{\bH}{\mathbb{H}}
\newcommand{\bI}{\mathbb{I}}
\newcommand{\bJ}{\mathbb{J}}
\newcommand{\bK}{\mathbb{K}}
\newcommand{\bL}{\mathbb{L}}
\newcommand{\bM}{\mathbb{M}}
\newcommand{\bN}{\mathbb{N}}
\newcommand{\bO}{\mathbb{O}}
\newcommand{\bP}{\mathbb{P}}
\newcommand{\bQ}{\mathbb{Q}}
\newcommand{\bR}{\mathbb{R}}
\newcommand{\bS}{\mathbb{S}}
\newcommand{\bT}{\mathbb{T}}
\newcommand{\bU}{\mathbb{U}}
\newcommand{\bV}{\mathbb{V}}
\newcommand{\bW}{\mathbb{W}}
\newcommand{\bX}{\mathbb{X}}
\newcommand{\bY}{\mathbb{Y}}
\newcommand{\bZ}{\mathbb{Z}}

\newcommand{\R}{\mathbb R}
\newcommand{\Q}{\mathbb Q}
\newcommand{\Z}{\mathbb Z}
\newcommand{\C}{\mathbb C}
\newcommand{\N}{\mathbb N}
\newcommand{\T}{\mathbb T}
\newcommand{\F}{\mathcal F}
\newcommand{\A}{\tilde{A}_{4,d}}
\newcommand{\Aq}{\tilde{A}_{q,d}}
\newcommand{\B}{\mathbb B}
\renewcommand{\S}{\mathcal S}
\newcommand{\w}{\omega}
\newcommand{\e}{\varepsilon}
\newcommand{\g}{\gamma}
\newcommand{\p}{\varphi}
\newcommand{\s}{\psi}
\newcommand{\z}{\zeta}
\renewcommand{\l}{\ell}
\renewcommand{\a}{\alpha}
\renewcommand{\b}{\beta}
\renewcommand{\k}{\kappa}
\newcommand{\m}{\textfrak{m}}
\renewcommand{\P}{\textfrak{p}}

\newcommand{\op}[1]{\operatorname{{#1}}}
\newcommand{\im}{\operatorname{Im}}
\newcommand{\pd}[2]{\frac{\partial #1}{\partial #2}}
\newcommand{\rd}[2]{\frac{d #1}{d #2}}
\newcommand{\ip}[2]{\langle #1 , #2 \rangle}
\renewcommand{\j}[1]{\langle #1 \rangle}
\newcommand{\cl}{\operatorname{cl}}
\newcommand{\into}{\hookrightarrow}

\newcommand{\mc}{\mathcal}
\newcommand{\mb}{\mathbb}
\newcommand{\f}{\mathfrak}
\newcommand{\ssp}{\sqsubseteq}
\newcommand{\cc}{\overline}
\renewcommand{\Re}{\text{Re}}
\renewcommand{\Im}{\text{Im}}


\def\bE{{\mathbf E}}
\def\bedagger{{\mathbf e}^\dagger}
\def\bF{{\mathbf F}}
\def\reals{\mathbb{R}}
\def\be{{\mathbf e}}
\def\symdif{\,\Delta\,}
\def\bH{{\mathbf H}}
\def\bV{{\mathbf V}}
\def\bG{{\mathbf G}}
\def\Psharp{P^\sharp}
\def\bA{{\mathbf A}}
\def\bI{{\mathbf I}}
\def\bEstar{{\mathbf E}^\star}
\def\bAstar{{\mathbf A}^\star}
\def\be{{\mathbf e}}
\def\bv{{\mathbf v}}
\def\bw{{\mathbf w}}
\def\br{{\mathbf r}}
\def\Star{\star}
\def\doublestar{\dagger\ddagger}
\def\unitQ{{\mathbf Q}}
\def\Gl{\operatorname{Gl}}
\def\eps{\varepsilon}
\def\naturals{{\mathbb N}}
\def\rplus{{\mathbb R}^+}
\def\scriptt{{\mathcal T}}
\def\integers{{\mathbb Z}}
\def\one{{\mathbf 1}}

\def\repair{\medskip\hrule\hrule\medskip}
\def\bff{\mathbf f}
\def\bE{\mathbf E}
\def\bx{\mathbf x}
\def\bu{\mathbf u}
\def\boldR{\mathbf R}
\def\by{\mathbf y}
\def\bz{\mathbf z}
\def\rationals{\mathbb Q}
\newcommand{\norm}[1]{ \|  #1 \|}
\def\scriptl{{\mathcal L}}
\def\scripti{{\mathcal I}}
\def\scriptr{{\mathcal R}}

\def\bEdagger{{\mathbf E}^\dagger}
\def\set4{\mathcal I}
\def\tup14{(1,2,3,4)}
\def\bk{\mathbf k}
\def\sl{\operatorname{Sl}}
\def\gl{\operatorname{Gl}}
\def\eps{\varepsilon}

\title{A symmetrization inequality shorn of symmetry}

\author{Michael Christ}

\address{
        Michael Christ\\
        Department of Mathematics\\
        University of California \\
        Berkeley, CA 94720-3840, USA}
\email{mchrist@berkeley.edu}

\author{Dominique Maldague}
\address{
        Dominique Maldague\\
        Department of Mathematics\\
        University of California \\
        Berkeley, CA 94720-3840, USA}
\email{dmal@math.berkeley.edu}
\thanks{The first author was supported in part by NSF grant DMS-1363324. 
The second author was supported by an NSF graduate research fellowship.}

\date{October 11, 2018}

\begin{abstract}
An inequality of Brascamp-Lieb-Luttinger and of Rogers
states that among subsets of Euclidean space $\reals^d$ of specified Lebesgue measures,
balls centered at the origin are maximizers of certain functionals
defined by multidimensional integrals. 
For $d>1$, this inequality only applies to functionals
invariant under a diagonal action of $\sl(d)$.
We investigate functionals of this type, and their maximizers,
in perhaps the simplest situation in which 
$\sl(d)$ invariance does not hold. 
Assuming a more limited symmetry involving dilations but not rotations,
we show under natural hypotheses that maximizers exist, 
and moreover, that there exist distinguished maximizers 
whose structure reflects this limited symmetry.
For small perturbations of the $\sl(d)$--invariant framework
we show that these distinguished maximizers 
are strongly convex sets with infinitely differentiable boundaries.
It is shown that maximizers fail to exist for certain arbitrarily
small perturbations of $\sl(d)$--invariant structures.
\end{abstract}

\maketitle

\section{Introduction}

Let $J$ be a finite index set, and for each $j\in J$ let
$L_j:\reals^D\to\reals^{d_j}$ be a surjective linear mapping.
Writing $\bff = (f_j: j\in J)$, consider the functional
$\bff\mapsto \Lambda(\bff)$ defined by
\begin{equation} \label{Lambdadefn1} 
\Lambda(\bff) = \int_{\reals^D} \prod_{j\in J} f_j(L_j(\bx))\,d\bx.
\end{equation}
The functions $f_j:\reals^{d_j}\to[0,\infty]$ are assumed
to be nonnegative and  Lebesgue measurable. 
The theory of H\"older-Brascamp-Lieb inequalities
\cite{liebgaussian},
\cite{BCCT1}, \cite{BCCT2},
\cite{carlenliebloss},
\cite{brascamplieb},
\cite{barthe},
\cite{wigdersonetal},
\cite{BBFL},
\cite{BBCF}
is concerned with inequalities 
$\Lambda(\bff)\le A\prod_{j\in J}\norm{f_j}_{L^{p_j}(\reals^{d_j})}$.
It includes a necessary and sufficient condition on
the data $D,J,d_j,L_j,p_j$ for there to exist
$A<\infty$ for which such an inequality holds for all $\bff$,
it provides an expression of sorts for the optimal constant $A$, 
it includes algorithms for computing certain elements
of the theory,
it has discrete variants which are closely connected with
Hilbert's tenth problem (over $\rationals$),
it includes a characterization of maximizing tuples $\bff$
under certain auxiliary hypotheses,
and the optimal constant 
$\underset{\bff}\sup \,\Lambda(\bff)/\underset{j\in J}\prod\norm{f_j}_{L^{p_j}}$
has been shown to be a H\"older
continuous function of $\scriptl=(L_j: j\in J)$ within an appropriate 
domain and under appropriate hypotheses.

One of the foundational instances of this theory concerns the Riesz-Sobolev functional
\begin{equation}
(f_1,f_2,f_3)\mapsto \langle f_1*f_2,f_3\rangle = \int_{\reals^d\times\reals^d} f_1(x)f_2(y)f_3(x+y)\,dx\,dy
\end{equation}
defined by pairing the convolution $f_1*f_2$ with $f_3$.
The Riesz-Sobolev inequality 
extends the conclusion beyond the H\"older-Brascamp-Lieb theory through the symmetrization inequality
\begin{equation}
\langle f_1*f_2,f_3\rangle \le \langle f_1^\star*f_2^\star,f_3^\star\rangle,
\end{equation}
where $f^\star:\reals^d\to[0,\infty)$
is (up to redefinition on Lebesgue null sets) the 
unique function that is radially symmetric,
is a nonincreasing function of $|x|$,
and is equimeasurable with $f$.
The general inequality is a direct consequence of the special case
in which each function $f_j$ is the indicator function $\one_{E_j}$
of a set.

In this paper, we are concerned with functionals \eqref{Lambdadefn1},
acting only on tuples of indicator functions of sets.
We abuse notation systematically by writing $\Lambda(\bE)$
for $\Lambda(\bff)$ where $f_j=\one_{E_j}$ and $\bE=(E_j: j\in J)$,
assuming always that each $E_j\subset\reals^{d_j}$ is a Lebesgue measurable subset
of $\reals^{d_j}$ with finite Lebesgue measure.
More accurately, each $E_j$ is an equivalence class of sets,
with $E$ equivalent to $E'$ if and only if $|E\symdif E'|=0$.

To $E\subset\reals^d$ is associated its symmetrization $E^\star\subset\reals^d$, 
defined to be the closed ball
whose Lebesgue measure equals that of $E$ if $|E|>0$,
and to be the empty set if $|E|=0$.
Define $\bEstar = (E_j^\star: j\in J)$.
Rogers \cite{rogers1}, \cite{rogers2}
and Brascamp-Lieb-Luttinger \cite{BLL}
have extended\footnote{The treatment of Rogers \cite{rogers2} for $d>1$ may be incomplete.} 
the Riesz-Sobolev symmetrization inequality to 
\begin{equation} \label{ineq:symm} \Lambda(\bE)\le\Lambda(\bEstar), \end{equation}
under certain natural hypotheses. Firstly, it is assumed that $d_j=d$ 
is independent of the index $j\in J$.
Secondly, $D/d=m\in\naturals$. 
If $d=1$ then \eqref{ineq:symm} holds under these hypotheses. 
If $d>1$ then \eqref{ineq:symm} holds 
under an additional symmetry hypothesis, under which there exists
an identification of $\reals^D = \reals^{md}$ with $(\reals^d)^m$
so that the diagonal action of $\sl(d)$ on $(\reals^d)^m$
is a symmetry of $\Lambda$, in the sense that
\begin{equation}
\Lambda(\bff)=\Lambda(\bff\circ T) \ \text{ for every $T\in\sl(d)$},
\end{equation}
where $\bff\circ T = (f_j\circ T: j\in J)$.
There is also a natural translation action of the additive group $\reals^{md}$
by $\by \mapsto \big(\bff \mapsto (f_j + L_j(\by): j\in J)\big)$,
under which $\Lambda$ is invariant.

The inequality \eqref{ineq:symm} for indicator functions can be read in two ways:
as a statement of monotonicity of $\Lambda$ under the mapping
$\bE \mapsto\bE^\star = (E_j^\star: j\in J)$,
or alternatively as a formula for the functional
\begin{equation}
\Theta(\be) = \sup_{|E_j|=e_j} \Lambda(\bE)
\end{equation}
where the supremum is taken over all tuples of measurable sets of
the specified Lebesgue measures.
In particular, \eqref{ineq:symm} states that maximizers of $\Theta$
exist, and that among these maximizers are tuples of balls centered at the origin
of the specified measures.
Consequently, according to the symmetry hypothesis,
tuples of homothetic ellipsoids whose centers belong to the orbit
of $0\in(\reals^{d})^J$ under the group of translation symmetries are also maximizers. 
This orbit is the set of all $|J|$--tuples $(L_j(\bv): j\in J)$, where $\bv$ ranges
over $\reals^{D}$. 

Uniqueness theorems \cite{burchard}, \cite{christflock}, \cite{christBLLeq}, \cite{christoneill}
state that these are the only maximizers, under certain additional hypotheses,
of which the primary one is known as admissibility \cite{burchard}.
These uniqueness theorems for indicator functions do not have simple
extensions to general nonnegative functions, 
yet they can sometimes be used to analyze uniqueness and stability questions for functionals
of general nonnegative functions \cite{christradon}, \cite{drouot}, \cite{christyoungest}.

In this paper, we take up the question of whether any part of
this theory for indicator functions survives
in the absence of the Rogers-Brascamp-Lieb-Luttinger symmetry hypothesis.
In general, ellipsoids are not maximizers, as this example reveals:
Let $J=\{0,1,2\dots,D\}$.
Let $d_j=1$ for every $j\ne 0$ and $d_0=D-1$.
For $1\le j\le d$ define $L_j(x_1,x_2,\dots,x_d)=x_j$.
Let 
$L_0:\reals^D\to\reals^{D-1}$ be a generic surjective linear mapping.
Let $E_j\subset\reals^1$ be the interval of length $1$ centered at $0$
for each $j\in\{1,2,\dots,D\}$. Let $E_0$ remain unspecified as yet.
$\Lambda(\bE)$ is equal to $\int_{E_0}K$
where $K:\reals^{D-1}\to[0,\infty)$ 
and $K(y)$ is the one-dimensional measure of the slice $\{\bx: L_0(\bx)=y\}$
of the unit cube in $\reals^D$.
For $|E_0|$ in a suitable parameter range,
maximizing sets $E_0$ are superlevel sets $\{y: K(y)\ge r\}$ of $K$, 
with $r$ a function of $|E_0|$. These superlevel sets are convex polytopes.

We study the equidimensional case in which $d_j=d$ for every index $j$.
We consider the simplest equidimensional situation not subsumed by existing theory:
$d=2$, $D=2d=4$, and the index set $J$ is $\set4=\{1,2,3,4\}$.
We impose a partial symmetry hypothesis, discussed below.
Our first two main conclusions concerning this situation are that there is a suitable generalization
of the concept of admissibility, and that maximizing tuples $\bE$ exist. 
This raises the question of the nature of such maximizers.
In the subcase in which the tuple $\scriptl$ of mappings $L_j$ is a small perturbation
of a tuple for which the symmetry hypothesis holds, we also show
that for any partially symmetrized maximizer $\bE$, each component set $E_j$ is strongly convex
with $C^\infty$ boundary.
Finally, we analyze a family of perturbed structures for which the partial symmetry is overtly broken
in a specific way, and show that maximizers $\bE$ exist for these structures 
if and only if they are equivalent via 
certain changes of coordinates in $\reals^4$ to structures with the partial symmetry.
Generically, such changes of coordinates do not exist.
Thus the partial symmetry condition is not wholly artificial.

Our partial symmetry hypothesis is most transparently expressed
in coordinates. For $\reals^4$, we use coordinates $(\bx;\by)=(x_1,x_2;y_1,y_2)$.
We assume that each target space $\reals^2$
is equipped with coordinates with respect to which the linear mapping $L_j:\reals^4\to\reals^2$
takes the form
\begin{equation} \label{PSH0} L_j(\bx,\by)=\big(L_{j}^1(\bx),L_{j}^2(\by)\big)\end{equation}
with $L_{j}^i:\reals^2\to\reals^1$ a surjective linear mapping. 
The perturbed structures of our nonexistence examples take the form
$L_j(\bx,\by)=\big(L_{j}^1(\bx),L_{j}^2(\bx,\by)\big)$.

Structures of the form \eqref{PSH0} enjoy two types of symmetries.
Firstly, there is a translation action of $\reals^4$ on $(\reals^2)^4$ defined by
\[(x_j: j\in \set4)\mapsto (x_j+ L_j(\bw)): j\in\set4)\]
for $\bw\in\reals^4$.
Secondly, there are dilation actions of $\reals^+$ on $\reals^2$
and on $(\reals^2)^4$, defined by
\begin{equation} \label{dilations:Dt} 
D_t(x,y) =(tx,t^{-1}y)\end{equation}
and $D_t(z_j: j\in\set4) = (D_t z_j: j\in\set4)$.

In the fully symmetric case, Steiner symmetrization \cite{liebloss}, \cite{BLL}  
and rotational symmetry combine to provide a powerful tool. 
Our partial symmetry hypothesis allows Steiner symmetrization with respect to
the horizontal and vertical axes, but not with respect to arbitrary directions
in $\reals^2$. This limited symmetrization is a useful tool, but certainly a less powerful one. 

An essential element in the theory of maximizers
in the fully symmetric situation is the notion of admissibility. In the Riesz-Sobolev
inequality, if $|E_3|^{1/d}>|E_1|^{1/d}+|E_2|^{1/d}$
then maximizing configurations are those in which 
the sumset $E_1+E_2$ has measure $\le|E_3|$ and is contained in $E_3$.
Thus maximizers exist, but have little structure
and are not a natural topic of discussion.
Admissibility for this inequality is the condition that
$|E_k|^{1/d}\le |E_i|^{1/d}+|E_j|^{1/d}$
for all permutations $(i,j,k)$ of $(1,2,3)$.
We formulate a suitable
definition of admissibility for our context, and combine Steiner symmetrization
with the translation and  dilation symmetries to develop a compactness argument
which establishes the existence of maximizers in the admissible regime.

We study in more detail those maximizers $\bE$ that are Steiner symmetric with respect to both the
horizontal and vertical axes and show that (under a certain auxiliary
hypothesis of genericity) each component set $E_j$
is strictly convex with $C^\infty$ boundary. This is a type of regularity theorem
for a coupled system of free boundary problems. The $4$--tuple $\bE$ satisfies
a generalized Euler-Lagrange relation, which states (formally) that
the boundary of each $E_i$ is a level set of a certain function
$K_i$ defined in terms of the other three sets $E_j$. 
A bootstrapping argument is used to establish $C^\infty$ regularity
along with strong convexity.


\section{Notation, hypotheses, and preliminaries}

Throughout the paper we write $\scripti= \{1,2,3,4\}$.
All sets $E_i\subset\reals^d$ are assumed to be Lebesgue measurable
and to have finite Lebesgue measures, unless otherwise indicated.

Consider functionals of the form  
\begin{equation} \Lambda_{\mc{L}}({\bf{E}})= \int_{\R^4}\prod_{i=1}^4 1_{E_i}(L_i(x_1,x_2,y_1,y_2))dx_1dy_1dx_2dy_2 \end{equation}
where ${\bf{E}}=(E_i:i\in\set4)$ is a 4--tuple of Lebesgue measurable subsets of $\R^2$ and $\mc{L}=(L_i:i\in\set4)$ is a collection of linear maps from $\R^4\to\R^2$. 
The following structural hypothesis on the maps $L_i$ will be in force throughout this paper: 
For each $i\in\set4$, we require that $L_i$ can be expressed in the form 
\begin{equation} \label{PSH} L_i(x_1,x_2,y_1,y_2)=(L_i^1(x_1,x_2),L_i^2(y_1,y_2)) \end{equation}
where $L_i^1:\R^2\to\R$ and $L_i^2:\R^2\to\R$ are linear and surjective. 
We refer to \eqref{PSH} as the partial symmetry hypothesis.

\begin{definition} \label{defn:fullRBLL}
A tuple $\scriptl^0 = (L_j^0: j\in\set4)$ is said to satisfy
the Rogers-Brascamp-Lieb-Luttinger symmetry hypothesis if satisfies \eqref{PSH} and
\begin{equation} \label{RBLLsymmetry}
L_i^1=L_i^2\ \text{ for each $i\in\scripti$.}
\end{equation}
\end{definition}
We say more succinctly that $\scriptl^0$ satisfies the full symmetry hypothesis.

This implies the presence of a large symmetry group.
Define $T(\bE) = (T(E_j): j\in\set4)$. Then \eqref{PSH} and \eqref{RBLLsymmetry} imply that 
\begin{equation*}
\Lambda(T(\bE))=\Lambda(\bE)\ \text{ for every $\bE$ and $T\in\sl(2)$.} 
\end{equation*} 

The following notion of nondegeneracy is equivalent to Definition 2.3 of \cite{christBLLeq} 
when $\mc{L}$ satisfies the full symmetry hypothesis. 

\begin{definition}\label{newnondegeneracy} 
A family $\mc{L}=(L_i:i\in\set4)$ of linear mappings $L_j:\R^4\to\R^2$ 
that satifies \eqref{PSH}
is nondegenerate if for any $i\ne j\in\scripti$,
the mappings $\bx\mapsto (L_i^1(\bx),L_j^1(\bx))$
and $\by\mapsto (L_i^2(\by),L_j^2(\by))$
are bijective linear transformations from $\reals^2$ to $\reals^2$.
\end{definition}

\begin{notation}
The Lebesgue measure preserving dilations $D_t:\reals^2\to\reals^2$ are defined by
\[D_t(x,y) = (tx,t^{-1}y)\]
for $t\in\reals^+$.
We also write \[D_t\bE = D_t(E_j: j\in\set4) = (D_tE_j: j\in\set4).\]
\end{notation}

These dilations are symmetries of $\Lambda_\scriptl$ in the sense that
\begin{equation}
\Lambda_{\scriptl}(D_t \bE) = \Lambda_{\scriptl}(\bE)
\end{equation}
for all $4$--tuples $\bE$ of sets of Lebesgue measurable subsets of $\reals^2$.
$\Lambda_\scriptl$ also enjoys a translation symmetry. For any $\bv\in\reals^4$,
$\Lambda(E_j: j\in\set4) = \Lambda(E_j+L_j(\bv): j\in\set4)$.
This follows by making a change of variables $(\bx,\by)\mapsto (\bx,\by)-\bv$
in the integral defining $\Lambda(\bE)$.

\begin{notation} 
$|E|$ denotes the Lebesgue measure of a subset of Euclidean space $\reals^d$ of any dimension $d$.
$|{\bf{E}}|$ denotes $(|E_1|,|E_2|,|E_3|,|E_4|)\in[0,\infty]^4$
where each $E_i$ is a Lebesgue measurable subset of $\reals^2$.
\end{notation}

\begin{notation}For ${\bf{e}}=(e_i:i\in\set4)\in(0,\infty)^4$, 
\[ \Theta(\be) := \sup_{\bE: |\bE|=\be} \Lambda_{\mc{L}}(\bE).\]
\end{notation}

\begin{lemma}
$\Theta$ satisfies a triangle inequality
\begin{equation}
\label{triangle}
\Theta(\be+\be') \ge \Theta(\be)+\Theta(\be').
\end{equation}
\end{lemma}

\begin{proof}
Consider any $\bE,\bE'$ satisfying $|\bE|=\be$
and $|\bE'|=\be'$ such that all of the component sets $E_j,E'_j$
are bounded. 
Choose a vector $\bv\in\reals^4$ that does not belong to
the nullspace of any of the four mappings $L_j$.
For large $r\in\reals^+$ consider the $4$--tuple $\bE^{(r)}$
of sets defined by $E^{(r)}_j = E_j \cup (E'_j + rL_j(\bv))$.
For sufficiently large $r$, $E'_j + rL_j(\bv)$ is disjoint from $E_j$,
so $|E^{(r)}_j| = e_j + e'_j$. 
Since $\one_{\tilde E_j} = \one_{E_j} + \one_{E'_j + rL_j(\bv)}$,
$\Lambda(\tilde\bE) \ge \Lambda(\bE) + \Lambda(E'_j+r L_j(\bv): j\in\set4)$.
Indeed, 
$\Lambda(\tilde\bE)$ is the sum of the two terms
on the right-hand side of this last inequality, plus $2^4-2$ other terms,
each of which is nonnegative.
By the translation invariance of $\Lambda$, 
this is equal to
\[\Lambda(\bE) + \Lambda(E'_j+r L_j(\bv): j\in\set4) = \Lambda(\bE) + \Lambda(\bE').\] 
Upon taking the supremum over all
tuples $\bE,\bE'$ with bounded component sets, the triangle inequality follows.
\end{proof}

The vertical Steiner symmetrization $\bE^\sharp = (E_j^\sharp: j\in\set4)$ is defined as follows. For $E\subset\reals^2$ with finite Lebesgue measure, $E^\sharp\subset\reals^2$
is $\{(x,y): |y|\le \tfrac12 |\{t\in\reals: (x,t)\in E|\}$
if $|\{t\in\reals: (x,t)\in E|\}>0$, and otherwise $\{y: (x,y)\in E^\sharp\}$ is empty.
Then $|E^\sharp|=|E|$, and the intersection of $E^\sharp$ with any vertical line
has the same one-dimensional Lebesgue measure as the intersection of $E$ with that same
vertical line. 
Define the horizontal Steiner symmetrizations $E^\flat$ and $\bE^\flat$ 
by interchanging the roles of the horizontal and vertical axes.
Define $E^\dagger = (E^\sharp)^\flat$ and $\bE^\dagger = (\bE^\sharp)^\flat$.
It is elementary that 
\begin{equation} E^\dagger=(E^\dagger)^\sharp = (E^\dagger)^\flat \end{equation}
up to Lebesgue null sets.

\begin{lemma}
$\Lambda$ satisfies
\begin{equation} \label{symmetrize}
\Lambda(\bE) \le\Lambda(\bEdagger) 
\end{equation}
for all tuples of sets of finite Lebesgue measure.
\end{lemma}

\begin{proof}

Under our partial symmetry hypothesis,
\begin{equation} \label{Steinersymminequality}
\Lambda(\bE)\le \Lambda(\bE^\sharp)
\ \text{ and } \ 
\Lambda(\bE)\le \Lambda(\bE^\flat)
\end{equation}
for arbitrary $\bE$. These inequalities are proved 
in \cite{BLL},
under the full Rogers-Brascamp-Lieb-Luttinger symmetry hypothesis of 
Definition~\ref{defn:fullRBLL}, 
but only the partial symmetry hypothesis is needed in their proofs
since only Steiner symmetrizations in horizontal and vertical directions are employed.
\eqref{symmetrize} follows from \eqref{Steinersymminequality} since 
\[ \Lambda(\bE)\le \Lambda(\bE^\sharp) \le\Lambda((\bE^\sharp)^\flat) = \Lambda(\bEdagger).\]
\end{proof}

\section{Admissibility}
We regard $(0,\infty)^4$ as being partially ordered. 

\begin{notation} $\be\le \be'$ means $e_j\le e'_j$ for all four indices $j\in\scripti$.
$\be < \be'$ means $\be\le \be'$ and $e_j<e'_j$ for at least one index $j\in\scripti$.
\end{notation}

\begin{definition}\label{admissible} 
$(\scriptl,\be)$ is admissible 
if there exists no $\be'<\be$ satisfying $\Theta(\be')=\Theta(\be)$.
\end{definition}
We will sometimes write ``$\be$ is admissible'' instead.

$\bE$ is said to be a maximizer if $\Lambda(\bE)=\Theta(|\bE|)$. 
$\Theta(\be)$ is said to be attained if there exists a maximizer with $|\bE|=\be$.

\begin{lemma}
$\Theta$ is locally Lipschitz continuous.
More precisely, there exists $C<\infty$ depending only on $\scriptl$
such that for any $\be,\be'\in(0,\infty)^4$,
\begin{equation} \big|\,\Theta(\be') - \Theta(\be)\,\big|
 \le C\max_{k\in\scripti} (e_k+e'_k)\max_{j\in\scripti} |e_j-e'_j|.\end{equation}
\end{lemma}

\begin{proof} 
The mapping $\bE\mapsto\Lambda_{\mc{L}}(\bE)$ 
is locally Lipschitz in the sense that
\begin{equation} \label{basicbound}
|\Lambda(\bE)-\Lambda(\bE')| \le C(\max_{i\in\scripti} |E_i|+ \max_{j\in\scripti}
|E'_j| \max_{k\in\scripti} |E_k\symdif E'_k|\end{equation}
for arbitrary $4$-tuples of Lebesgue measurable subsets of $\R^2$. 
This constant $C$ depends only on $\scriptl$.

Given $\be$ and $\delta>0$,
choose $\bE=(E_j: j\in\scripti)$ satisfying $|\bE|=\be$ and $\Lambda(\bE)\ge \Theta(\be)-\delta$.
From the sets $E_j$, construct sets $E'_j\subset\reals^2$ satisfying $|E'_j|=e'_j$
with $|E'_j\symdif E_j| = |e'_j-e_j|$. 
It follows from \eqref{basicbound} that
\[\Lambda(\bE')\ge \Lambda(\bE)- C\max_{k\in\scripti} e_k \cdot \max_{j\in\scripti} |e_j-e'_j|,\]
where $C<\infty$ depends only on $\scriptl$.
By letting $\delta\to 0$ we conclude that
\[ \Theta(\be') \ge \Theta(\be) - C\max_{k\in\scripti} (e_k+e'_k)\max_{j\in\scripti} |e_j-e'_j|.\]
\end{proof} 

The mapping $\scriptl\mapsto\Lambda_\scriptl(\bE)$ is not 
continuous in $\scriptl$ {\em uniformly in $\bE$}.

\begin{proposition} 
For any $t>0$ there exists an admissible $\be\in(0,\infty)^4$ satisfying $\Theta(\be)=t$.
For any $\be\in(0,\infty)^4$ 
there exists an admissible $\be'\le\be$ satisfying 
$\Theta(\be') = \Theta(\be)$.
\end{proposition}

\begin{proof}
$\Theta(r\be)=\Theta(re_1,re_2,re_3,re_4)=r^2\Theta(\be)$
for any $r\in(0,\infty)$ and $\be\in(0,\infty)^4$.
Therefore for any $t\in(0,\infty)$ there exists $\tilde\be$ satisfying $\Theta(\tilde\be)=t$.

Let $t>0$.
Choose $A$ so that there exists $\bE$ satisfying $\Lambda_{\mc{L}}(\bE)=t$ 
with $|E_j|\le A$ for each $j$.
Let $S$ be the set of all $\be$ satisfying $\Theta(\be)=t$
and $e_j\le A$ for each $j$. The choice of $A$ ensures that $S\ne\emptyset$.
It is a consequence of the continuity of $\Theta$ that $S$ is closed. 
Since the nondegeneracy hypothesis ensures that 
\[ \Lambda(\bE) \le C |E_i|\cdot|E_j| \text{ for any $i\ne j\in\scripti$}\]
where $C<\infty$ depends only on $\scriptl$,
it follows that $\inf_{\be\in S}\min_{i\in\scripti} e_i$ is strictly positive.
Thus $S$ is a compact subset of the open upper quadrant.

Let $\bar e_1 = \underset{\be\in S}\min \, e_1$. Define
\[\bar e_2 = \underset{\substack{\be\in S\\ e_1={\bar e_1}}}\min\, e_2 \]
and iterate this process to define $\bar e_3$ and then $\bar e_4$.
Because $S$ is compact, these quantities $\bar e_i$ exist. Because $S$ is closed, 
$\bar{\be}=(\bar e_1,\bar e_2,\bar e_3,\bar e_4)$ lies in $S$. 
The construction guarantees that there exists no $\be\in S$ satisfying $\be<\bar{\be}$. 
\end{proof}

To any $\scriptl$, any ordered tuple $(E_j,E_k,E_l)$, and any index $i$ such that $\{i,j,k,l\}=\set4$ 
is associated a unique function
$K_i: \reals^2\to[0,\infty)$ characterized by the relation
\begin{equation} \label{Kidefn}
\Lambda(E_1,E_2,E_3,E_4) = \langle \one_{E_i},K_i\rangle = \int K_i \one_{E_i}
\ \text{ for every $E_i\subset\reals^2$.}
\end{equation}

If $\bE^0$ is a $4$--tuple of balls in $\reals^2$ centered at the origin, and if $\scriptl=\scriptl^0$
satisfies the full symmetry hypothesis of Definition~\ref{defn:fullRBLL}
then the associated quantities $K_i^0$ are radially symmetric for each $i\in\set4$.

\begin{definition}
Let $\scriptl^0$ satisfy the 
full symmetry hypothesis of Definition~\ref{defn:fullRBLL}.
Let $\bE^0$ be a $4$--tuple of balls in $\reals^2$ centered at the origin and let $\be=|\bE^0|$. 
$(\scriptl^0,\be)$ is strictly admissible if for each $i\in\set4$,
$K_i^0>0$ in some neighorhood of $\partial E_i^0$,
and $\frac{d}{du}K_i^0(u^-,0)<0$,
where $u\in(0,\infty)$ is defined by the property that $(u,0)$ belongs to
the boundary of $E_i^0\subset\reals^2$.
\end{definition}

The notation $\frac{d}{du}K_i^0(u^-,0)$ denotes
the one-sided derivative \[\lim_{h\to 0^-} h^{-1} \big( K_i^0(u+h,0)-K_i^0(u,0) \big).\]

It is shown in \cite{christoneill}
that if $\scriptl^0$ is nondegenerate\footnote{This statement is
proved for $d\ge 2$ in \cite{christoneill}. The corresponding statement for $d=1$, with a
supplementary genericity hypothesis, is proved in \cite{christBLLeq}.} 
and satisfies the full symmetry hypothesis of Definition~\ref{defn:fullRBLL},
and if $(\scriptl^0,\be)$ is strictly admissible,
then every maximizer $\bE$ satisfying $|\bE|=\be$ for $\Lambda_{\scriptl^0}$
is in the orbit of a $4$--tuple of balls centered at the origin in $\reals^2$
under the symmetry group generated by translations and by the diagonal action of $\sl(2)$.
Thus each $E_i$ is an ellipse, these ellipses are homothetic, and the the $4$--tuple of their centers
belongs to the orbit of $(0,0,0,0)\in (\reals^2)^4$ under the translation symmetry group.

$K_i$ can be written in the form
\begin{equation} \label{Kiintegral}
K_i(u) = \int_{\reals^2} \prod_{j\ne i} \one_{E_j}(\ell_{j,i}(u,v))\,dv
\end{equation}
where $\ell_{j,i}:\reals^4\to\reals^2$ are surjective linear maps
of the form 
\begin{equation} \label{Kiintegralform}
\ell_{j,i}(u,v) = \big(\ell_{j,i,1}(u_1,v_1),\ell_{j,i,2}(u_2,v_2)\big)
\end{equation}
with $\ell_{j,i,m}:\reals^2\to\reals^1$ linear and surjective.
Moreover, $\ell_{j,i}-\ell^0_{j,i} = O(\norm{\scriptl-\scriptl^0})$
and $\ell_i^0 = (\ell_{j,i}^0: j\in\set4\setminus\{i\})$
commute with the diagonal action of $\sl(2)$.
$\{\ell_{j,i}: j\ne i\}$ is nondegenerate in the sense that
for any two distinct indices $j,k\ne i$,
the linear mapping $(u,v)\mapsto (\ell_{j,i}(u,v),\ell_{k,i}(u,v))$
from $\reals^2\times\reals^2$ to $\reals^2\times\reals^2$ is nonsingular.

An extra condition which had not previously appeared in the theory
for the fully symmetric framework arises naturally in our analysis. 
Formulation of this condition requires some additional notation.
Let $\scriptl^0$ be nondegenerate and satisfy the full symmetry hypothesis 
of Definition~\ref{defn:fullRBLL}.
Let $(\scriptl^0,\be)$ be strictly admissible.
Let $\bE^0=(E_j^0: j\in\scripti)$ 
be a tuple of balls centered at the origin of size $|\bE^0|=\be$. For any $i\ne j\in\set4$, 
and any $w\in\reals^2$,
define the sets $\tilde E_j(w)\subset\reals^2$ by 
\[\one_{E_j^0}(\ell_{j,i}(w,v)) = \one_{\tilde E_j(w)}(v).\]
These sets are balls, whose centers are $\reals^2$--valued
linear functions of $w\in\reals^2$; the mappings from $\reals^2\owns w$
to their centers are radially symmetric functions.

Write $\scripti = \{i,j,k,l\}$, with $i$ playing the same role as in 
\eqref{Kiintegral} and \eqref{Kiintegralform}.
Let $w=(u,0)\in\partial E_i^0$ with $u>0$.
The Lebesgue measure of $\tilde E_j(u,0)\cap\tilde E_k(u,0)\cap\tilde E_l(u,0)\subset\reals^2$ 
equals $K_i^0(u,0)$, which is strictly positive by the strict admissibility hypothesis.
The centers of these closed three balls $\tilde E_j(u,0),\tilde E_i(u,0),\tilde E_l(u,0)$
lie on the horizontal axis in $\reals^2$.
Strict admissibility implies that none of these three balls is contained in the interiors
of the other two. 

\begin{definition} \label{genericity}
Let $\scriptl^0$ be nondegenerate and satisfy the full symmetry hypothesis
of Definition~\ref{defn:fullRBLL}, and let $(\scriptl^0,\be)$ be strictly admissible.
Let $\bE^0$ be a $4$--tuple of balls centered at $0$ satisfying
$|E_j^0|=e_j$, let $u\in(0,\infty)$ be defined by $(u,0)\in \partial E_i^0$, and for $w\in\reals^2$ let $\tilde E_j(w)$ be associated to $E_j$ as above.
$(\scriptl^0,\be)$ is generic if one of the following two mutually exclusive cases holds:
\newline
(i) After some permutation of $(j,k,l)$, $\tilde E_j(u,0)\cap \tilde E_k(u,0)$
is contained in the interior of $\tilde E_l(u,0)$,
and the boundary of $\tilde E_j(u,0)\cap \tilde E_k(u,0)$
consists of a subarc of the boundary of $\tilde E_j(u,0)$ and a subarc 
of the boundary of $\tilde E_k(u,0)$, meeting transversely at two points.
\newline
(ii) The threefold intersection $\tilde E_j(u,0)\cap\tilde E_k(u,0)\cap\tilde E_l(u,0)$ is a connected,
simply connected domain whose boundary is a piecewise $C^\infty$
curve consisting of $4$ subarcs of circles, with two of these arcs contained in the boundary
of one of the three closed balls, exactly one of the arcs contained in the boundary of another of the three closed balls, the final arc contained in the boundary of the remaining closed ball, 
and with arcs meeting transversely where they intersect on the boundary.
\end{definition}

The case that is excluded by this definition of genericity 
is that in which, after permutation of $(j,k,l)$,
$\tilde E_l(u)$ contains $\tilde E_j(u)\cap\tilde E_k(u)$, but the interior of $\tilde E_l(u)$
does not contain this intersection.
In this situation, the boundary of the three-fold intersection consists of one subarc
of $\partial\tilde E_j(u)$ and one subarc of $\partial\tilde E_k(u)$, meeting transversely,
but the points at which these subarcs meet also belong to $\partial\tilde E_l(u)$.
This situation is unstable, giving rise to either case (i) or (ii) upon arbitrary small
perturbation. We expect this instability to lead to failure 
of the boundary of components $E_i$ of maximizing tuples $\bE$ 
to be $C^\infty$, for general perturbations $\scriptl$ of $\scriptl^0$.

The following example shows that the property that ${\bf{e}}$ is generic does not follow from ${\bf{e}}$ being strictly admissible: Let ${\bf{E}}$ be a $4$-tuple of subsets of $\R^2$ and consider the functional 
\begin{align}
 \int_{\R^4} 1_{E_1}(x_1,y_1)1_{E_2}(x_2,y_2)1_{E_3}(x_1+x_2,y_1+y_2)1_{E_4}(x_1-x_2,y_1-y_2)d{\bf{z}} 
\end{align} 
written in coordinates ${\bf{z}}=(x_1,x_2,y_1,y_2)$. The data ${\bf{e}}=(1,1,r,1)$ is strictly admissible for the above functional if $1<r<2$. Because the functional satisfies the symmetry hypotheses of the Rogers-Brascamp-Lieb-Luttinger inequality, the sets $(B,B,B_r,B)$, where $B$ is the radius 1 ball centered at the origin and $B_r$ is the radius $r$ ball centered at the origin, extremize the functional restricted to tuples of sets of size $|{\bf{E}}|=(1,1,r,1)$. Using the notation from the previous definition with $i=1$, we have $u=1$, $\tilde{E_1}=\tilde{E}_4=B$ and $\tilde{E}_3=B_r$. The intersection
\[ B\cap(B_r-(1,0))\cap (B+(1,0)) \]
varies in type, satisfying case (i) or (ii) from Definition \ref{genericity}, and sometimes neither, for different values of $1<r<2$, as shown in Figure \ref{fig:intersections} below.  
\begin{figure}[h!]
    \centering
\begin{tabular}{ cc }

\begin{tikzpicture}
\draw[step=.5cm,gray,very thin] (-4,4) (4,4);
\draw [color=blue](0,0) circle (1cm);
\draw [color=red](1,0) circle (1cm);
\draw [color=green](-1,0) circle (1.3cm);
\draw (0.3,1.7) node {$r=1.3$}; \end{tikzpicture} & \begin{tikzpicture}
\draw[step=.5cm,gray,very thin] (-4,4) (4,4);
\draw [color=blue](0,0) circle (1cm);
\draw [color=red](1,0) circle (1cm);
\draw [color=green](-1,0) circle (1.73205080757cm);
\draw (1,1.5) node {$r=\sqrt{3}$};
\end{tikzpicture} \\
  \multicolumn{2}{c}{\begin{tikzpicture}
\draw[step=.5cm,gray,very thin] (-4,4) (4,4);
\draw [color=blue](0,0) circle (1cm);
\draw [color=red](1,0) circle (1cm);
\draw [color=green](-1,0) circle (1.85cm);
\draw (1.3,1.4) node {$r=1.85$};
\end{tikzpicture}
} \\

\end{tabular}
    \caption{The balls that form $B\cap(B_r-(1,0))\cap (B+(1,0))$ for different values of $r$. The intersection is of type (i) for $r=1.3$, type (ii) for $r=1.85$, and neither of type (i) nor of type (ii) for $r=\sqrt{3}$. }
    \label{fig:intersections}
\end{figure}
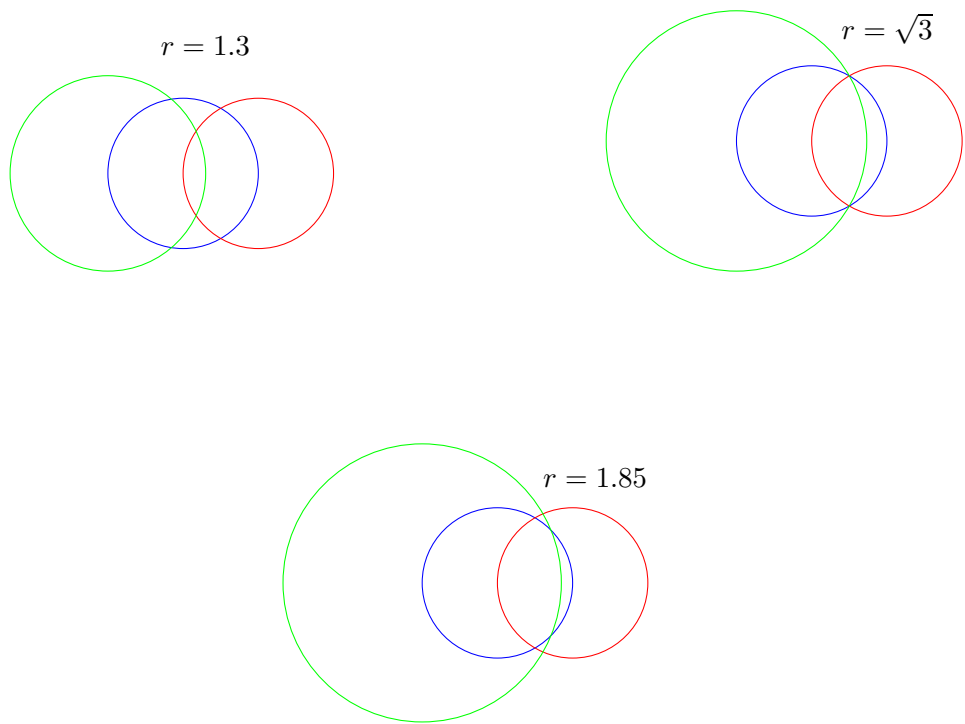

\section{Main results}

The main results of this paper are as follows.

\begin{theorem}[Existence of maximizers] \label{thm:existence}
Let $\scriptl$ be nondegenerate and satisfy the partial symmetry hypothesis \eqref{PSH}.
For each $\be\in(0,\infty)^4$
there exists $\bE$ satisfying $|\bE|=\be$ and $\Lambda(\bE) = \Theta(\be)$.
\end{theorem}

It will suffice to prove Theorem~\ref{thm:existence} in the admissible case.
For we have shown that
if $(\scriptl,\be)$ is not admissible, then there exists an admissible $\be'\le\be$
satisfying $\Theta(\be')=\Theta(\be)$.
If there exists $\bE'$ satisfying $|\bE'|=\be'$
and $\Lambda(\bE')=\Theta(\be')=\Theta(\be)$,
then any tuple $\bE$ satisfying $E_j\supset E'_j$ for each $j\in\set4$
and $|\bE|=\be$ is a maximizer for $(\scriptl,\be)$.

For admissible $(\scriptl,\be)$ a stronger result will be proved.

\begin{theorem}[Qualitative stability of maximizers] \label{thm:precompact} 
Let ${\bf{e}}\in(\R^+)^4$ be admissible. Let ${\bf{E}}^{(n)}$ be a sequence of tuples 
satisfying 
$|{\bf{E}}^{(n)}|={\bf{e}}$, $\underset{n\to\infty}\lim \Lambda({\bf{E}}^{(n)}) = \Theta({\bf{e}})$, 
and $(\bE^{(n)})^\dagger = \bE^{(n)}$. 
Then there exist a subsequence of indices $n_k$, 
real numbers $\lambda_k>0$, and a tuple ${\bf{E}}$ such that for each index $i\in\scripti$, 
\begin{equation} \lim_{k\to\infty} |E_i\symdif D_{\lambda_k}(E_i^{n_k})| = 0.  
\end{equation} 
\end{theorem}

Theorem~\ref{thm:existence} for admissible $(\scriptl,\be)$ is a direct consequence of 
Theorem~\ref{thm:precompact}. Indeed, let $\be$ be admissible.
According to the definition of $\Theta(\be)$ as a supremum
and by virtue of the symmetrization inequality $\Lambda(\bE)\le\Lambda(\bE^\dagger)$,
there exists a sequence $(\bE_\nu: \nu\in\naturals)$
of tuples 
satisfying $\bE_\nu = \bE_\nu^\dagger$, $|\bE_\nu|=\be$, 
and $\underset{\nu\to\infty}\lim \Lambda(\bE_\nu) = \Theta(\be)$. 
Since $\Lambda$ and the relation $|\bE|=\be$ are invariant under the group of dilations
$D_t$, the continuity of $\Lambda$ implies that the tuple $\bE=(E_i: i\in\set4)$
whose existence is guaranteed by Theorem~\ref{thm:precompact}
satisfies $\Lambda(\bE) = \lim_{\nu\to\infty} \Lambda(\bE_\nu) = \Theta(\be)$. Theorem~\ref{thm:precompact} is proved in \textsection\ref{precmpct}. 

In a perturbative regime we obtain information about symmetrized maximizers $\bE$.

\begin{theorem}[Convexity and regularity] \label{thm:bootstrap}
Let $\scriptl^0$ satisfy the full symmetry hypothesis
of Definition~\ref{defn:fullRBLL}.
Suppose that $(\scriptl^0,\be)$ is strictly admissible and generic.
There exists $\delta>0$ with the following property.
Let $\scriptl$ satisfy the limited symmetry hypothesis
and satisfy $\norm{\scriptl-\scriptl^0}<\delta$.
Let $\bE$ satisfy $|\bE|=\be$ and be a maximizer for $\Lambda_\scriptl$.
Suppose that $\bE^\dagger=\bE$. 
Then for each $j\in\set4$,
 $E_j$ is a strongly convex set with $C^\infty$ boundary.
\end{theorem}

Strong convexity of a $C^2$ domain means that the domain is convex, and that its boundary has nonzero curvature at every point. Theorem~\ref{thm:bootstrap} is proved in \textsection\ref{perturb}.

A final result 
indicates that the partial symmetry hypothesis \eqref{PSH} is not entirely artificial.

\begin{theorem} \label{thm:nonexistence}
Let $\scriptl^0$ satisfy the full symmetry hypothesis \eqref{RBLLsymmetry}.
There exist nondegenerate $\scriptl$ arbitrarily close to $\scriptl^0$
such that for any $\be$ such that $(\scriptl^0,\be)$
is strictly admissible, there exist no maximizers $\bE$
for $\Lambda_\scriptl$ satisfying $|\bE|=\be$.
\end{theorem}

Proposition~\ref{prop:nonexistence}
states more specifically that for tuples of mappings $\scriptl$ of the form
$L_j(\bx,\by)=\big(L_{j}^1(\bx),L_{j}^2(\bx,\by)\big)$, 
maximizers $\bE$ cannot exist unless 
$L_{j}^2(\bx,\by)$ takes a special form
which makes the functional $\Lambda_\scriptl$ equivalent,
in a natural way, to $\Lambda_{\tilde\scriptl}$
where $\tilde\scriptl$ satisfies \eqref{PSH}.
It is proved in Section~\ref{section:nonexistence}.

\section{Conjectures}

In this paper we explore a rather specific situation 
in the hope of building insight into what is true, and what might be proved,
in a far broader framework.
It is natural to venture various conjectures in this regard.
In all of these conjectures, we assume that $(\scriptl,\be)$ satisfies
appropriate nondegeneracy and admissibility hypotheses, whose formulations
are perhaps not yet known.

\begin{conjecture} 
For generic $(\scriptl,\be)$, 
maximizers $\bE$ are not tuples of ellipsoids.
\end{conjecture}

\begin{conjecture} 
For generic $(\scriptl,\be)$ satisfying the partial symmetry hypothesis \eqref{PSH},
any maximizer $\bE$ is a translate of a maximizer that satisfies $\bE = \bEdagger$.
\end{conjecture}

\begin{conjecture} 
Let $(\scriptl^0,\be)$ be nondegenerate and strictly admissible.
For generic admissible $(\scriptl,\be)$ satisfying \eqref{PSH}
with $\scriptl$ sufficiently close to $\scriptl^0$,
maximizers of $\Lambda_\scriptl$ that satisfy $\bE = \bEdagger$
are unique up to measure-preserving dilations of $\reals^2$.
\end{conjecture}

\begin{conjecture} 
Under the partial symmetry hypothesis \eqref{PSH},
the conclusion that the component sets $E_i$ of any maximizer $\bE$ are
convex, holds with suitable strict admissibility and nondegeneracy
hypotheses on $(\scriptl,\be)$, without any hypothesis
that $\scriptl$ is a small perturbation of a tuple of mappings that possesses $\sl(d)$ invariance.
\end{conjecture}

\begin{conjecture} 
The results of this paper concerning existence and convexity of maximizers 
have analogues for generic nondegenerate data $\scriptl$ without partial symmetry.
\end{conjecture}

Proposition~\ref{prop:nonexistence} demonstrates that the requirement that $\scriptl$
be generic cannot be entirely omitted. However, the construction on which the Proposition
is based requires structural properties not shared by generic $\scriptl$, and we
do not regard these examples as indicative of the state of affairs for generic data.

\begin{conjecture}
The $C^\infty$ boundary conclusion of Theorem~\ref{thm:bootstrap} 
need not hold, in general, for small perturbations $\scriptl$, satisfying \eqref{PSH},
of tuples $\scriptl^0$ that satisfy \eqref{RBLLsymmetry},
if the genericity hypothesis on $(\scriptl^0,\be)$ is omitted.
\end{conjecture}

\begin{question} 
To what extent are maximizers $\bE$ unique up to translation,
in the absence of partial symmetry, for generic $(\scriptl,\be)$
satisfying suitable nondegeneracy and admissibility hypotheses?
\end{question}

\section{Compatibility} \label{section:compatible}

\begin{definition}
For tuples of sets $\bE = (E_j: j\in\set4)$,
\begin{equation}\label{lambdadefn} \lambda(\bE) = \sup_{R}
\min_{j\in\set4} \frac{|E_j\cap R|}{|E_j|+|R|} \end{equation}
with the supremum taken over all 
rectangles $R\subset\reals^2$ 
centered at $0$ with sides parallel to the coordinate axes.
\end{definition}

If $E=E^\dagger$, and if $R$ is a rectangle centered at the origin with sides
parallel to the coordinate axes, then
\begin{equation} \label{ok}
|(R+y)\cap E|\le |R\cap E| \ \text{ for every $y\in\reals^2$.}
\end{equation}

\begin{lemma}[Compatibility] \label{lemma:rectangle2}
Let $K$ be a compact subset of $(\reals^+)^4$. 
For each  $\eps>0$ there exists $\delta>0$ with the following property.
Let $\be\in K$, and let
$\bE$ be a $4$--tuple of subsets of $\reals^2$ satisfying $|\bE|=\be$ and $\bE=\bEdagger$.
If $\lambda(\bE)<\delta$ then $\Lambda(\bE)<\eps$.
\end{lemma}

\begin{sublemma} \label{lemma:simplest4bound}
For $j\in\scripti$
let $R_j=I_j\times I'_j\subset\reals^2$ be a rectangle with sides parallel to
the axes. 
Then for any permutation $(i,j,k,l)$ of $(1,2,3,4)$,
\begin{equation} \label{eq:simplest4bound}
\Lambda(\one_{R_j}: j\in\set4)
\le C |I'_k|\cdot|I'_l|\cdot|I'_i|^{-1}|I'_j|^{-1} |R_i|\cdot|R_j|.
\end{equation}
\end{sublemma}

\begin{proof}
$\Lambda(\one_{R_j}: j\in\set4)$
is majorized by the Lebesgue measure of the set of all
$(\bx,\by) = (x_1,x_2,y_1,y_2)\in\reals^4$ for which
$L_i^1(\bx)\in I_i$,
$L_j^1(\bx)\in I_j$,
$L_k^2(\by)\in I'_k$,
and
$L_l^2(\by)\in I'_l$.
The mappings $\bx\mapsto (L_i^1(\bx),L_j^1(\bx))$
and
$\by\mapsto (L_k^2(\by),L_l^2(\by))$
are bijective linear transformations from $\reals^2$
to $\reals^2$.
Thus
$\Lambda(\one_{R_j}: j\in\set4)$
is bounded by a constant, which depends only on $(L_n: n\in\set4)$,
multiplied by $|I_i|\cdot|I_j|\cdot |I'_k|\cdot |I'_l|$. 
\end{proof}
An analogous conclusion holds if the roles of $I_k$ and $I_k'$ are reversed.

In the following discussion, $\bk = (k_1,k_2,k_3,k_4)\in\integers^4$.
\begin{sublemma} \label{lemma:Ssum}
For each $j\in\set4$ let $\{R_{k}^{(j)}: k\in\integers\}$
be a family of rectangles  in $\reals^2$
of the form $R_k^{(j)} = I_k^{(j)}\times J_k^{(j)}$
with $J_k^{(j)}$ of length $2^k$.
Suppose that 
$\underset{k\in\integers}\sum |R_k^{(j)}|<\infty$ for each $j\in\set4$.
There exists $C<\infty$ such that
for any set $S\subset\integers^4$,
\begin{multline*}
\sum_{\bk\in S} \Lambda(R_{k_1}^{(1)}, R_{k_2}^{(2)},R_{k_3}^{(3)},R_{k_4}^{(4)})
\le C 
\sup_{\bk\in S}\,\, 
2^{-(\underset{i}\max \,k_i - \underset{j}\min\, k_j)/2}\cdot \sup_{\bk\in S} \max_{n\in\set4} |R_{k_n}^{(n)}|
\cdot
\max_{j\in\set4} \sum_{k\in\integers} |R_k^{(j)}|.
\end{multline*}
\end{sublemma}

\begin{proof}
By dilating we may assume without loss of generality that
\[\sup_{j\in\set4} \sum_{k\in\integers} |R_k^{(j)}|=1.\]
It suffices to treat the summation over all $\bk\in S$ that satisfy
\begin{equation} \label{eq:ordering}
k_4\le k_3\le k_2\le k_1.
\end{equation}
The same reasoning will apply with arbitrary permutations of the indices $1,2,3,4$.
For the rest of the proof, we assume that every $\bk\in S$ satisfies
\eqref{eq:ordering}.
Set $\rho = 
\underset{\bk\in S}\sup \,\,2^{(k_4+k_3-k_2-k_1)/2}$.

Individual summands satisfy
\begin{align*} 
\Lambda(R_{k_1}^{(1)},R_{k_2}^{(2)},R_{k_3}^{(3)},R_{k_4}^{(4)})
&\le C 2^{k_3+ k_4-k_1-k_2}
|R_{k_1}^{(1)}|\cdot
|R_{k_2}^{(2)}| 
\\&
\le C\rho 
2^{-(k_1-k_2)/2}
2^{(k_3+ k_4-2k_2)/2}
|R_{k_1}^{(1)}|\cdot
|R_{k_2}^{(2)}|.
\end{align*}
Summation over all $k_3,k_4\le k_2$  yields an upper bound
\[
\rho \sum_{k_2\le k_1} 2^{-(k_1-k_2)/2} 
|R_{k_1}^{(1)}|\cdot
|R_{k_2}^{(2)}| 
\le C\rho\big(\sum_{k_1} |R_{k_1}^{(1)}|^2\big)^{1/2}
\cdot \big(\sum_{k_2} |R_{k_2}^{(2)}|^2\big)^{1/2}
\]
with the first sum taken over all $(k_1,k_2)$ satisfying $k_2\le k_1$
for which there exist $k_3,k_4$ for which $\bk\in S$, 
the second sum over all $k_1$ for which
there exist $k_2,k_3,k_4$ for which $\bk\in S$, 
and so on.
Now 
\[ \big(\sum_{k_1} |R_{k_1}^{(1)}|^2\big)^{1/2}
\le \sup_{k} |R_k^{(1)}|^{1/2}
\cdot \big(\sum_{k_1} |R_{k_1}^{(1)}|\big)^{1/2}
\le \sup_{k} |R_k^{(1)}|^{1/2},\]
with a corresponding majorization for $\big(\sum_{k_2} |R_{k_2}^{(2)}|^2\big)^{1/2}$.
\end{proof}

\begin{proof}[Proof of Lemma \ref{lemma:rectangle2}] 
Define $E_j^+=\{(x,y)\in E_j: x>0\}$
and
$E_j^-=\{(x,y)\in E_j: x<0\}$.
We will analyze $\Lambda(E_j^+: j\in\set4)$;
the same reasoning will apply equally well to $\Lambda(E_j^\pm: j\in\set4)$
with all possible choices of $\pm$ signs.

To $E_j$ associate  rectangles $R_k^{(j)}\subset\reals^2$
with sides parallel to the coordinate axes, defined as follows:
Express $E_j^+$,
up to a Lebesgue null set,
as 
\[ E_j^+=\{(x,y): x>0 \text{ and }  |y|<f_j(x)\}\] where $f_j:(0,\infty)\to[0,\infty)$  
is nonincreasing and right continuous.
Define 
\[ R_k^{(j)}=\{x\in\reals^+:  2^k\ge f_j(x) > 2^{k-1}\} \times[-2^k,2^k].\]

Then $E_j^+\subset \cup_{k=-\infty}^\infty R_k^{(j)}$,
so $|E_j|\le 2\underset{k}\sum |R_k^{(j)}|$.
On the other hand,
$|R_k^{(j)}\cap E_j|\ge \tfrac12 |R_k^{(j)}|$, so
\[\sum_{k\in\integers} |R_k^{(j)}| \le 2|E_j^+| =|E_j|.\]

Express
\[ \Lambda(E_j^+: j\in\set4)
= \sum_{\bk\in\integers^4} \Lambda(E_j^+\cap R_{k_j}^{(j)}: j\in\set4)
\le \sum_{\bk\in\integers^4} \Lambda(R_{k_j}^{(j)}: j\in\set4).  \]
Let $\eta = \eta(\delta)>0$ be a small parameter which will be chosen below to depend only on 
$\delta$, and will tend to zero as $\delta\to 0$.
Introduce
\begin{align*}
S_1  &= \{\bk\in S:  \max_{j\in\set4} |R_{k_j}^{(j)}|\le\eta\}
\\
S_2 &= \{\bk\in S\setminus S_1: \max_{i\in\set4} k_i-\min_{j\in\set4} k_j > \log_2(1/\eta)\}
\\
S' &= S\setminus (S_1\cup S_2).
\end{align*}
By Sublemma~\ref{lemma:Ssum},
\[\Lambda(E_j^+: j\in\set4)
\le C\eta^{1/2} + C\sum_{\bk\in S'} \Lambda(R_{k_j}^{(j)}: j\in\set4). \]
Matters are thus reduced to the sum over $\bk\in S'$.

As above, by partitioning $S'$ into finitely many subsets,
we may assume for the remainder of the proof that
$|R_{k_1}^{(1)}| \ge |R_{k_2}^{(2)}| \ge |R_{k_3}^{(3)}| \ge |R_{k_4}^{(4)}|$ for each $\bk\in S'$.
Since each $R_{k_j}^{(j)}$ 
is a rectangle with sides parallel to the coordinate axes
with vertical side of length $2^{k_j}$,
and since $\underset{i,j\in\set4}\max\, 2^{k_i}/2^{k_j}\le \eta^{-1}$,
\eqref{eq:simplest4bound} gives
\begin{align*}  
\Lambda(R_{k_j}^{(j)}: j\in\set4)
&\le C 2^{k_1}2^{k_3}2^{-k_4}2^{-k_2}|R_{k_4}^{(4)}|
\cdot |R_{k_2}^{(2)}|\le C\eta^{-2} \frac{|R_{k_4}^{(4)}|}{ |R_{k_1}^{(1)}|}
\cdot |R_{k_1}^{(1)}| 
\cdot |R_{k_2}^{(2)}|
\end{align*}
Since $\bk\not\in S_2$, $k_3\le k_2+\log_2(1/\eta)$, and summing over such $k_3$ yields the bound
\[ C\eta^{-1} 2^{k_1}2^{-k_4}|R_{k_4}^{(4)}|
\cdot |R_{k_2}^{(2)}| .\]
Summing over $k_2$ gives 
\[ C\eta^{-1}2^{k_1}2^{-k_4}|R_{k_4}^{(4)}|. \]
Again since $\bk\not\in S_2$, $-k_4\le \log_2(1/\eta)-k_1$. Summing over all remaining indices which in addition satisfy  $\underset{i\in\set4}\min |R_{k_i}^{(i)}| \le \eta^3 \underset{j\in\set4}\max |R_{k_j}^{(j)}|$
gives the bound
\begin{align*}
C\eta^{-1}\sum_{k_1}&\sum_{2^{k_1-k_4}\le 1/\eta} 2^{k_1-k_4}\left(|R_{k_4}^{(4)}|/|R_{k_1}^{(1)}|\right)|R_{k_1}^{(1)}| \\
&\le C\eta^{2}\sum_{k_1}\sum_{2^{k_1-k_4}\le 1/\eta} 2^{k_1-k_4}|R_{k_1}^{(1)}|\\ 
&\le C\eta\sum_{k_1}|R_{k_1}^{(1)}|=C\eta .
\end{align*} 
Therefore it remains to consider indices $\bk\in S'$ which satisfy $\underset{i\in\set4}\min |R_{k_i}^{(i)}|> \eta^3 \underset{j\in\set4}\max |R_{k_j}^{(j)}|$.

Assume that $\lambda(\bE)<\delta$, with $\delta>0$ small.
Let $\tilde S$ be the set of all $\bk\in S$ that remain untreated,
and for which $\Lambda(R_{k_j}^{(j)}: j\in\set4)\ne 0$.
To complete the proof, it suffices to show that
if $\eta(\delta)\to 0$ sufficiently
slowly as $\delta\to 0$, then the hypothesis that $\lambda(\bE)<\delta$
forces $\tilde S$ to be empty. 

Consider any $\bk\in\tilde S$.
The associated four rectangles $R_{k_j}^{(j)}$ have sides parallel
to the coordinate axes, have vertical sides of comparable lengths ---
meaning that the ratios of any two of these lengths are bounded above
by a function of $\eta$ alone --- and have comparable Lebesgue measures,
in the same sense of comparability.
Therefore the lengths of  their horizontal sides are likewise comparable.
Therefore there exists a single rectangle $R\subset\reals^2$,
with sides parallel to the coordinate axes, such that $R_{k_j}^{(j)}$ is
contained in a translate $R+\bv_j$ and has Lebesgue measure comparable to $|R|$,
for each $j\in\set4$.
Since $|R_{k_j}^{(j)}\cap E_j|\ge \tfrac12|R_{k_j}^{(j)}|$,
and since $|E_j|$ is comparable to $1$,
the ratio $\frac{|E_j\cap R_{k_j}^{(j)}|}{|E_j| + |R_{k_j}^{(j)}|}$
is comparable to $|R_{k_j}^{(j)}|$. 
Therefore we find that for each $j\in\set4$,
\[ |R|\le C(\eta) |R_{k_j}^{(j)}|\le C(\eta)\frac{|E_j\cap R_{k_j}^{(j)}|}{|E_j|+|R_{k_j}^{(j)}|}\le C(\eta)\frac{|E_j\cap R|}{|E_j|+|R|} ,\]
where $C(\eta)<\infty$ depends only on $\eta$ and we used \eqref{ok} in the last inequality. By the definition of $\lambda(\bE)$, we have
\[ |R|\le C(\eta) \lambda(\bE)\le C(\eta)\delta .\]
Therefore if 
$C(\eta)\cdot\delta <\eta$ then we conclude that $\bk\in S_1$,
whence $\bk\notin\tilde S$. Thus $\tilde S$ would be empty.

For each $\eta>0$, the inequality $C(\eta)\delta<\eta$
holds for all sufficiently small $\delta>0$.
Therefore there exists a function $\delta\mapsto\eta(\delta)$
satisfying both 
$\lim_{\delta\to 0} \eta(\delta)=0$,
and  
$C(\eta(\delta))\cdot\delta <\eta(\delta)$ for every $\delta>0$.
\end{proof}

\begin{corollary} \label{cor:centralsquare}
For each compact set $K\subset (\reals^+)^4$
and each $\delta>0$ 
there exists $\rho>0$ with the following property.
If $|\bE|=\be\in K$,
if $\bE = \bEdagger$,
and if $\Lambda(\bE)\ge\delta$
then there exists $t\in\reals^+$
such that for each $j\in\set4$, 
\begin{equation} [-\rho,\rho]\times[-\rho,\rho] \subset D_t(E_j).\end{equation}
\end{corollary}

\begin{proof}
This is a direct consequence of Lemma~\ref{lemma:rectangle2}
and the definition of $\lambda(\bE)$.
Choose a rectangle $R\subset\reals^2$, with sides parallel to the coordinate axes and centered
at the origin, that maximizes the ratio defining $\lambda(\bE)$, up to a factor of $2$.
The lower bound for $\Lambda(\bE)$ implies a lower bound for the Lebesgue
measure of $R$. An appropriate dilation gives $\rho = c |R|^{1/2}$
where $c>0$ is a constant.
\end{proof}

\section{Precompactness \label{precmpct}}

In this section we apply the results of Section~\ref{section:compatible} to 
establish Theorem~\ref{thm:precompact},
concerning the precompactness of symmetrized maximizing sequences for admissible $(\scriptl,\be)$
up to the dilation and translation symmetries of $\Lambda$ introduced above.
A pivotal issue is how the admissibility of $(\scriptl,\be)$
comes into play in the proof. 
As was implicitly shown in the comment following Theorem~\ref{thm:existence},
precompactness cannot hold if $(\scriptl,\be)$ is not admissible,
for if $\be'<\be$ with $e'_j<e_j$,
and if $\bE'=(E'_i: i\in\scripti)$ satisfies $\Lambda(\bE')=\Theta(\be')
=\Theta(\be)$ then any tuple $(E_i: i\in\scripti)$
with $E_i=E'_i$ for every $i\ne j$, 
$E_j\supset E'_j$, and $|E_j|=e_j$
satisfies $\Lambda(\bE)=\Theta(\be)$ and $|\bE|=\be$.

\begin{proof}[Proof of Theorem~\ref{thm:precompact}] 
Suppose that $|\bE^{(n)}|=\be$ for each $n\in\naturals$,
that $\bE = \bE^\dagger$,
and that $\Lambda(\bE^{(n)})\to\Theta(\be)$ as $n\to\infty$.
Write $\bE^{(n)} = (E_j^n: j\in\scripti)$.
By invoking Corollary~\ref{cor:centralsquare} and replacing each $\bE^{(n)}$ with a suitable dilate,
we may assume that there exists a cube $\tilde Q$ with positive sidelength,
centered at the origin, that is contained in $E_i^n$ for every $n$ and every $i\in\set4$.

By the Banach-Alaoglu theorem, there exists a subsequence $n_k$ of indices and functions $g_i\in L^2(\R^2)$ such that for every $i\in\scripti$, $1_{E_i^{n_k}}$ converges weakly in $L^2$ to $g_i$ as $k\to\infty$. 
By replacing ${\bf{E}}^{(n)}$ by a subsequence, we may assume henceforth that 
the full sequence of indicator functions $\one_{E_i^n}$ converges weakly in $L^2$.

The set $E_i^n$ intersected with $(0,\infty)\times[0,\infty)$ is the region under the
graph of a nonnegative, nonincreasing function $f_{i,n}$. 
For any $s>0$, $sf_{i,n}(s) \le \tfrac14 |E_i^n|=\tfrac14 e_i$.
A simple consequence of the Helly section theorem is that
the weak limit of $(\one_{E_i^n}: n\in\naturals)$ 
is the indicator function of a region 
\[ E_i = \{(u,v) :|v|\le f_i(u)\}\]
where $f_i$ is even, the restriction of $f_i$ to $(0,\infty)$
is nonincreasing, and $uf_i(u)\le e_i/4$ for every $u>0$.
Thus $g_i=\one_{E_i}$.

Set 
\begin{gather*}
\bE =(E_1,E_2,E_3,E_4),
\\
{\bf{E}}^{(n)}\cap {\bf{E}} = (E_i^n\cap E_i: i\in\set4),
\\
{\bf{E}}^{(n)}\setminus {\bf{E}} = (E_i^n\setminus E_i: i\in\set4).
\end{gather*} 

\begin{lemma}\label{lemma:claimmain'}
\begin{equation} \label{claim:main'}
    \lim_{n\to\infty} \Lambda({\bf{E}}_i^{(n)}) 
=\lim_{n\to\infty} [\Lambda({\bf{E}}^{(n)}\cap {\bf{E}})+\Lambda({\bf{E}}^{(n)}\setminus{\bf{E}}) ].  
\end{equation}
\end{lemma}

\begin{proof}[Proof of Lemma~\ref{lemma:claimmain'}]
Expressing the indicator function of $E_i^{(n)}$ as the sum of the 
indicator functions of $E_i^n\cap E_i$ and of $E_i^n\setminus E_i$,
and then invoking the multilinearity of $\Lambda$, 
produces an expansion of $\Lambda(\bE^{(n)})$ as a sum of $2^4$ terms, of which two are
the main terms
$\Lambda({\bf{E}}^{(n)}\cap {\bf{E}})$ and
$\Lambda({\bf{E}}^{(n)}\setminus {\bf{E}})$.
Each of the remaining $14$ terms takes the form
$\Lambda(E_1^n\cap E_1,E_2^n\setminus E_2,F_3^n,F_4^n)$ 
with $F_j^n$ equal either to $E_j^n\cap E_j$
or to $E_j^n\setminus E_j$,
up to permutation 
of the indices $1,2,3,4$. 
Moreover, 
\[\Lambda(E_1^n\cap E_1,E_2^n\setminus E_2,F_3^n,F_4^n) 
\le \Lambda(E_1^n\cap E_1,E_2^n\setminus E_2,E_3^n,E_4^n).\]
Thus in order to prove \eqref{claim:main'}, it will suffice to show that 
\[\Lambda(E_1^n\cap E_1,E_2^n\setminus E_2,E_3^n,E_4^n)\to 0\] as $n\to\infty$,
provided that the same reasoning applies with the indices permuted, as it indeed will.

To analyze $\Lambda(E_1^n\cap E_1,E_2^n\setminus E_2,E_3^n,E_4^n)$, 
let $\varepsilon>0$. For $R>0$, let $Q_R=[-R,R]^2$. For $N,M>0$, we have the upper bound
\begin{align*} 
\Lambda(E_1^n\cap E_1,E_2^n\setminus E_2,E_3^n,E_4^n)
&\le \Lambda(E_1\cap Q_M,E_2^n\setminus (E_2\cup Q_N),E_3^n,E_4^n)\\
&\quad\qquad +\Lambda(E_1\cap Q_M,(E_2^n\setminus E_2)\cap Q_N,E_3^n,E_4^n)
+C|E_1\setminus Q_M| e_2 
\\&\le \Lambda(E_1\cap Q_M,E_2^n\setminus Q_N,E_3^n,E_4^n)\\
&\qquad \quad+\Lambda(E_1,(E_2^n\setminus E_2)\cap Q_N,E_3^n,E_4^n)
+C|E_1\setminus Q_M| e_2. 
\end{align*}

We claim that for any $M<\infty$,
\begin{equation}
\Lambda(E_1\cap Q_M,E_2^n\setminus Q_N,E_3^n,E_4^n)
\le \rho_{M,\be}(N)   
\end{equation}
where the function $\rho_{M,\be}(N)$ depends only on $\scriptl,\be,M,N$
and 
$\rho_{M,\be}(N)\to 0$ as $N\to \infty$ while $M,\be,\scriptl$ remain fixed.
Indeed, define $\alpha$ so that 
the intersection of $E^n_2$ with $\{N\}\times\reals$,
which is an interval, has length $2\alpha$.
Then $(-N,N)\times(-\alpha,\alpha)\subset E_2^n$, so $\alpha \le 4N^{-1}e_2$.
Likewise, defining $\beta$ so that
the intersection of $E^n_2$ with $\reals\times \{N\}$
has length $2\beta$, one has $\beta\le 4N^{-1}e_2$.
Therefore if $N$ is sufficiently large, $E_2^n\setminus Q_N$
is contained in the union of $\reals\times(-\alpha,\alpha)$ with $(-\beta,\beta)\times\reals$.
Define $E^n_{2,h}$ to be the former portion of $E_2^n$, and $E^n_{2,v}$ to be the latter.
Consider
$\Lambda(E_1\cap Q_M,E_{2,h}^n\setminus Q_N,E_3^n,E_4^n)$,
which is majorized by
$\Lambda(E_1\cap Q_M,\tilde E_2^n,E_3^n,E_4^n)$,
where $\tilde E_2^n = (E_{2,h}^n)^\dagger$,
is the horizontal Steiner symmetrization\footnote{By the horizontal Steiner symmetrization
of $E\subset\reals^2$ we mean the set $F\subset\reals^2$
obtained by replacing $E^y = \{x\in\reals^1: (x,y)\in E\}$
by $[-E^y/2,E^y/2]$ whenever $|E^y|>0$, and by $\emptyset$
when $|E^y|=0$.}
of $E_{2,h}^n$.

We will apply Lemma~\ref{lemma:rectangle2},
which asserts that $\Lambda(\bE)$ is small if the quantity $\lambda(\bE)$ 
defined in \eqref{lambdadefn} is small.
Consider any rectangle $R\subset\reals^2$ with sides parallel to the coordinate axes
and centered at $0$.
In evaluating $\lambda(\bE)$, clearly only rectangles $R$ whose vertical sides
have length $\lesssim\alpha$ need be considered.
If $R$ does have vertical length $\lesssim\alpha$
then $R$ has small measure unless its horizontal side has length $\gtrsim \alpha^{-1}e_2
\gtrsim N/4$.
However, in this case $|R\cap E_1|\le |R\cap Q_M| \lesssim\alpha M
\lesssim MN^{-1}e_2$.
Therefore $\lambda(E_1\cap Q_M, \tilde E_2^n,E_3^n,E_4^n)$ becomes arbitrarily small
as $N$ becomes arbitrarily large. Therefore by Lemma~\ref{lemma:rectangle2}, the same goes for
$\Lambda(E_1\cap Q_M, E_{2,h}^n,E_3^n,E_4^n)$.
The same reasoning applies to $\Lambda(E_1\cap Q_M, E_{2,v}^n,E_3^n,E_4^n)$.

Choose $M$ sufficiently large that $|E_1\setminus Q_M|<\varepsilon$. 
Then choose $N$ large enough so that $\rho_{M,\be}(N)<\eps$.
Finally, the weak convergence of $E_2^n$ implies that
\[ |(E_2^n\setminus E_2)\cap Q_N|=|E_2^n\cap (Q_N\setminus E_2)|\to |E_2\cap( Q_N\setminus E_2)|= 0
\ \text{ as $n\to\infty$.} \]
Thus 
$\underset{n\to\infty}\limsup\, \Lambda(E_1^n\cap E_1,E_2^n\setminus E_2,E_3^n,E_4^n)\le 2\varepsilon$. 
Since $\eps>0$ was arbitrary, this proves that
$\Lambda(E_1^n\cap E_1,E_2^n\setminus E_2,E_3^n,E_4^n)\to 0$.
The same reasoning, with natural changes in notation,
proves that the other cross terms in the expansion of $\Lambda({\bf{E}}^{(n)})$ also have limit zero.
This completes the proof of 
Lemma~\ref{lemma:claimmain'}.
\end{proof}

We claim next that
\begin{equation} \label{claim:second} 
\lim_{n\to\infty} \Lambda(({\bf{E}}^{(n)}\setminus {\bf{E}}))\to 0.
\end{equation}
It suffices to show that
$\lim_{n\to\infty} \Lambda(({\bf{E}}^{(n)}\setminus {\bf{E}})^\star)\to 0$,
where $({\bf{E}}^{(n)}\setminus {\bf{E}})^\star = ((E_i^n\setminus E_i)^\star: i\in\set4)$ 
denotes the symmetrization of ${\bf{E}}^{(n)}\setminus {\bf{E}}$. 

To prove \eqref{claim:second}, note 
that since $|E_i^n\cap E_i|\to |E_i|$ by the weak convergence, 
\[\lim_{n\to\infty}|({\bf{E}}^{(n)}\setminus {\bf{E}})^*|=(e_1-|E_1|,e_2-|E_2|,e_3-|E_3|,e_4-|E_4|).\]  
If $\underset{n\to\infty}\limsup\, \Lambda(({\bf{E}}^{(n)}\setminus {\bf{E}})^*)>0$ then 
after passing to a subsequence of indices $n$ that realizes the limit supremum,
we may invoke Corollary~\ref{cor:centralsquare} to conclude that there exist a sequence of dilations
$D_{\eta_k}$ and $\delta>0$ such that the cube $[-\delta,\delta]^2=Q_\delta$ is contained in $ D_{\eta_k}((E_i^{k}\setminus E_i)^*)$ for all $i,k$. 

By a change of variables which preserves the partially symmetric multilinear structure of $\Lambda$ (permitted by Definition \ref{newnondegeneracy}), we may assume without loss of generality 
that $L_1(x_1,y_1,x_2,y_2)=(x_1,y_1)$ and $L_2(x_1,y_1,x_2,y_2)=(x_2,y_2)$. 
For each $k$, let $N_k>0$ be large enough so $|E_1^{n_k}\setminus Q_{N_k}|<\frac{1}{k}$ 
and $|D_{\eta_k}((E_1^{n_k}\setminus E_1)^*)\setminus  Q_{N_k}|<\frac{1}{k}$. 
Also let $v_1^k\in \R^2$ be large enough so that $Q_{N_k}\cap (Q_{N_k}+v_1^k)=\emptyset$. 
Then 
\begin{align*}
\Lambda({\bf{E}}^{(n_k)}\cap {\bf{E}})&+\Lambda({\bf{E}}^{(n_k)}\setminus{\bf{E}})
\le \Lambda(E_1^{n_k}\cap E_1 \cap Q_{N_k},E_2^{n_k}\cap E_2,E_3^{n_k}\cap E_3,E_4^{n_k}\cap E_4)\\
    &\quad+\Lambda(D_{\eta_k}(E_1^{n_k}\setminus E_1)^* \cap Q_{N_k},D_{\eta_k}(E_2^{n_k}\setminus E_2)^*,D_{\eta_k}(E_3^{n_k}\setminus E_3)^*,D_{\eta_k}(E_4^{n_k}\setminus E_4)^*) \\
& \qquad\qquad + Ce_2k^{-1} \\
&= \Lambda(E_1^{n_k}\cap E_1 \cap Q_{N_k}+v_1^k,E_2^{n_k}\cap E_2,E_3^{n_k}\cap E_3+v_3^k,E_4^{n_k}\cap E_4+v_4^k)\\
    &\quad+\Lambda(D_{\eta_k}(E_1^{n_k}\setminus E_1)^* \cap Q_{N_k},D_{\eta_k}(E_2^{n_k}\setminus E_2)^*,
D_{\eta_k}(E_3^{n_k}\setminus E_3)^*,D_{\eta_k}(E_4^{n_k}\setminus E_4)^*) \\
& \qquad\qquad + Ce_2k^{-1}
\end{align*}
for certain $v_3^k,v_4^k\in\reals^2$ determined by $v_1^k$ and $\scriptl$.
Thus
\begin{align*}
\Lambda({\bf{E}}^{(n_k)}\cap {\bf{E}})&+\Lambda({\bf{E}}^{(n_k)}\setminus{\bf{E}})
\le  \Theta(e_1,|(E_2^{n_k}\cap E_2)\cup D_{\eta_k}(E_2^{n_k}\setminus E_2)^*|,e_3,e_4)+Ce_2k^{-1}. 
\end{align*}

There exists a cube $\tilde Q$ centered at $0\in\reals^2$,
of positive sidelength, that is contained in $E_2$
and in $E_2^{n_k}$ for every $k$.
Therefore for every sufficiently large $k$,
\[|(E_2^{n_k}\cap E_2)\cup D_{\eta_k}(E_2^{n_k}\setminus E_2)^*|\le e_2-|\tilde{Q}\cap Q_\delta|<e_2.  \] 
By letting $k\to\infty$ one deduces that
\[ \Theta(e_1,e_2,e_3,e_4)\le \Theta(e_1,e_2-|\tilde{Q}\cap Q_\delta|,e_3,e_4).\]
This contradicts the definition of admissibility of ${\bf{e}}$. 
Therefore \eqref{claim:second} must hold.

\medskip
Combining \eqref{claim:second} with \eqref{claim:main'},
we find that $\Lambda(\bE^{(n)}\cap\bE)\to\Theta(\be)$ as $n\to\infty$.
By the admissibility of ${\bf{e}}$, this forces $\underset{n\to\infty}\liminf\, 
|E_i^n\cap E_i|=e_i$ for each $i\in\set4$.
Since $e_i = |E_i^n|$,
that completes the proof of Theorem~\ref{thm:precompact}.
\end{proof} 


\section{Continuity of $\Theta$ with respect to $\scriptl$}

In the next lemma, $\scriptl^0$ is not assumed to satisfy the full symmetry hypothesis
\eqref{RBLLsymmetry}, even though this notation is reserved for that special case
in nearly all of this paper. 

\begin{lemma} \label{lemma:continuityWRTL}
Let $\scriptl^0$ and $\scriptl_\nu$ 
be nondegenerate and satisfy \eqref{PSH}. 
Let $((\scriptl_\nu,\be_\nu): \nu\in\naturals)$ 
be a sequence of data such that $(\scriptl_\nu,\be_\nu)\to (\scriptl^0,\be^0)$
as $\nu\to\infty$.
Then
\begin{equation} \lim_{\nu\to\infty} \Theta_{\scriptl_\nu}(\be_\nu)
= \Theta_{\scriptl^0}(\be^0).\end{equation}
\end{lemma}

\begin{proof}
There exists
a symmetrized maximizing configuration $\bE^0$ for $(\scriptl^0,\be^0)$.
Modifying each component $E^0_i$ appropriately yields
a sequence $\bE_\nu$ of symmetrized $4$--tuples satisfying $|\bE_\nu|=\be_\nu$.
Then $|E_{\nu,i}\symdif E^0_i|\to 0$ as $\nu\to\infty$ for each $i\in\scripti$.
It follows that
$\Lambda_{\scriptl_\nu}(\bE_\nu)\to\Lambda_{\scriptl^0}(\bE^0) =\Theta_{\scriptl^0}(\be^0)$.
Therefore
\begin{equation}
\limsup_{\nu\to\infty} \Theta_{\scriptl_\nu}(\be_\nu)\ge \Theta_{\scriptl^0}(\be^0).
\end{equation}
However, no converse inequality follows with comparable ease,
because the mapping $(\scriptl,\bE)\mapsto\Lambda_\scriptl(\bE)$ fails to be continuous
in any sufficiently uniform sense with respect to $\bE$.

To prove the converse,
pass to a subsequence to ensure that $\Theta_{\scriptl_\nu}(\be_\nu)$
converges to $\Theta_{\scriptl^0}(\be^0)$ as $\nu\to\infty$.
Let $(\bE_\nu)$ satisfy $|\bE_\nu|=\be_\nu$
and $\limsup_{\nu\to\infty} \Lambda_{\scriptl_\nu}(\bE_\nu)
=\limsup_{\nu\to\infty} \Theta_{\scriptl_\nu}(\be_\nu)$,
and let each $\bE_\nu$ satisfy
$\bE_\nu = \bE_\nu^\dagger$.
Write $\bE_\nu = (E_{\nu,i}: i\in\scripti)$.
By replacing $\bE_\nu$ by $D_{t_\nu}\bE_\nu$ for an appropriately
chosen sequence of parameters $t_\nu\in(0,\infty)$,
we may assume that there exists $\rho>0$ such that
$[-\rho,\rho]^2\subset E_{\nu,i}$ for each $i\in\scripti$.

By repeating the reasoning in the proof of Theorem~\ref{thm:precompact}
we conclude that after passing to a subsequence, there exists
$\bE^{0,\sharp}=(E^{0,\sharp}_{i}: i\in\scripti)$,
satisfying $\bE^{0,\sharp} = (\bE^{0,\sharp})^\dagger$,
such that for each $i\in\scripti$,
$E_{\nu,i}$ may be expressed as a disjoint union
$E_{\nu,i}^\sharp \cup E_{\nu,i}^\flat$,
satisfying
$E_{\nu,i}^\sharp =(E_{\nu,i}^\sharp)^\dagger$
and
$E_{\nu,i}^\flat=(E_{\nu,i}^\flat)^\dagger$,
with $|E_{\nu,i}^\sharp\symdif E^0_{i,1}|\to 0$,
and
\[ \lim_{\nu\to\infty}
\big( \Lambda_{\scriptl_\nu}(\bE_\nu^\sharp)
+ \Lambda_{\scriptl_\nu}(\bE_\nu^\flat) \big)
= \lim_{\nu\to\infty} \Theta_{\scriptl_\nu}(\be_\nu),\]
where $\bE_\nu^\sharp = (E_{\nu,i}^\sharp: i\in\scripti)$
and analogously for $\bE_\nu^\flat$.
The cross terms that arise in the proof of Theorem~\ref{thm:precompact}
contribute zero in the limit $\nu\to\infty$,
because the bounds in Lemma~\ref{lemma:rectangle2}
are uniform in $\nu$ since $\scriptl_\nu\to\scriptl^0$.

Since $|E_{\nu,i}^\sharp \symdif E^{0,\sharp}_i|\to 0$
and $\scriptl_\nu\to\scriptl^0$ as $\nu\to\infty$,
\begin{equation}
\lim_{\nu\to\infty}
\Lambda_{\scriptl_\nu}(\bE_\nu^\sharp)
= \Lambda_{\scriptl^0}(\bE^{0,\sharp}). \end{equation}

Therefore setting $\be_{\nu,1}=\be^\flat_\nu$ and $\be^0_1 = \be^{0,\sharp}$,
\begin{equation} \limsup_{\nu\to\infty} \Theta_{\scriptl_\nu}(\be_\nu)
\le \Theta_{\scriptl^0}(\be^0_1) 
+ \limsup_{\nu\to\infty} \Theta_{\scriptl_\nu}(\be^\flat_{\nu,1})\end{equation}
and
\begin{equation} \be^0 = \be^0_1 + \lim_{\nu\to\infty} \be_{\nu,1}\end{equation}
with addition and limits defined componentwise for $4$--tuples.
If $\lim_{\nu\to\infty} \Theta_{\scriptl_\nu}(\be_{\nu,1})=0$
then the proof is complete.

Write $\be_{\nu,1}=(\be_{\nu,1,i}: i\in\scripti)$.
Observe that $\min_{i\in\scripti} |E^0_i|\ge 4\rho^2$
since $E_{\nu,i}\supset [-\rho,\rho]^2$ for every $\nu$.
Therefore for every sufficiently large $\nu$,
$e_{\nu,1,i}\le e^0_i-3\rho^2$.

Pass from the full sequence to a subsequence of indices $\nu$, along which
$\limsup_{\nu\to\infty} \Theta_{\scriptl_\nu}(\be_{\nu,1})$
is achieved in the limit.
Apply the above construction to 
obtain a partition of $\bE_{\nu,1}$ in terms of $\bE_{\nu,1}^\sharp$ and $\bE_{\nu,1}^\flat$
and a limiting set $\bE^{0,\sharp}_2 = \bE^0_2$.
Conclude in the same way that
\begin{equation}
\lim_{\nu\to\infty} \Theta_{\scriptl_\nu}(\be_\nu)
\le 
\Theta_{\scriptl^0}(\be^0_1)
+ \Theta_{\scriptl^0}(\be^0_2)
+ 
\limsup_{\nu\to\infty}
\Theta_{\scriptl_\nu}(\be_{\nu,2})\end{equation}
where $\be_{\nu,2} = \be^\flat_{\nu,1}$, with
\begin{equation} \lim_{\nu\to\infty}
\big(\be^0_1+\be_0^2+\be_{\nu,2}\big) =  \be^0 .  \end{equation}
As in the initial step, each component $e^0_{2,i}$ of the tuple $\be^0_2$
is minorized by a positive quantity, which in turn is minorized by
a positive function of $\lim_{\nu\to\infty} \Theta_{\scriptl_\nu}(\be_{\nu,2}))$.

Iterate this process. It may halt after finitely many steps, in which case it produces
a finite sequence $\be^0_k$ satisfying $\sum_k \be^0_k\le \be^0$
and $\limsup_{\nu\to\infty} \Theta_{\scriptl^0}(\be_\nu) = \sum_k \Theta_{\scriptl^0}(\be^0_k)$.
Obviously
$\sum_k \Theta_{\scriptl^0}(\be^0_k) \le \Theta_{\scriptl^0}(\be^0)$,
completing the proof.

If the process fails to halt after finitely many steps then
it produces infinite sequences  $\be^0_k$ and  $\be_{\nu,k}$.  
Necessarily 
\begin{equation}
\lim_{k\to\infty} \limsup_{\nu\to\infty} \Theta_{\scriptl_\nu}(\be_{\nu,k}) = 0. 
\end{equation}
Indeed, at each step there exists $\rho_k>0$
for which $[-\rho_k,\rho_k]^2\subset E^{0,\sharp}_i$ for each $i\in\scripti$,
and $\rho_k$ is bounded below by a strictly positive quantity
if $\limsup_{\nu\to\infty} \Theta_{\scriptl_\nu}(\be_{\nu,k})$ 
is bounded away from zero.
If $\limsup_{\nu\to\infty}\Theta_{\scriptl_\nu}(\be_{\nu,k})$
did not tend to zero then each component of $\be^0_k$
would be bounded away from zero uniformly in $k$,
contradicting the relation $\sum_k \be^0_k\le \be^0$.

Therefore
\begin{equation} \limsup_{\nu\to\infty} \Theta_{\scriptl_\nu}(\be_\nu)
\le \sum_{k=0}^\infty \Theta_{\scriptl^0}(\be^0_k)\end{equation}
with $\sum_k \be^0_k\le \be^0$.
This last inequality  implies that
$\sum_{k=0}^\infty \Theta_{\scriptl^0}(\be^0_k)
\le \Theta_{\scriptl^0}(\be^0)$,
once more completing the proof.
\end{proof}

\begin{corollary} \label{cor:strictsublinearity}
Let $\be',\be''\in(0,\infty)^4$. Let $\scriptl$
be nondegenerate and satisfy the partial symmetry hypothesis \eqref{PSH}.
If $(\scriptl,\be'+\be'')$ is admissible then 
\begin{equation}
\Theta_\scriptl(\be')+\Theta_\scriptl(\be'') < \Theta_\scriptl(\be'+\be'').
\end{equation}
\end{corollary}

\begin{proof}
First suppose that $\be',\be''$ are admissible.
Let $\bE',\bE''$ be associated symmetrized maximizing tuples.
As in the proof of Theorem~\ref{thm:precompact},
there exists $c>0$ such that 
for any $\eps>0$
there exists $\bv\in \reals^4$
such that the tuple $\bE = (E'_i\cup (E''_i+\scriptl_i(\bv)): i\in\scripti)$
satisfies $\Lambda(\bE) \ge \Lambda(\bE')+\Lambda(\bE'')-\eps$
and $|E'_1\cup (E''_1+\scriptl_1(\bv))|\le e'_1+e''_1-c$.
Thus $\Theta(e'_1+e''_1-c,e'_2+e''_2,\dots) \ge \Theta(\be'+\be'')$,
contradicting the admissibility of $\be'+\be''$.
\end{proof}

\begin{lemma}
Let $(\scriptl^0,\be^0)$ and $(\scriptl_\nu,\be_\nu)$ 
satisfy the hypotheses of Lemma~\ref{lemma:continuityWRTL}.
Suppose that $(\scriptl^0,\be^0)$ satisfies
the full symmetry hypothesis of Definition~\ref{defn:fullRBLL}, and is strictly admissible.
Suppose that $|\bE_\nu|=\be_\nu$,
that $\bE_\nu = \bE_\nu^\dagger$,
and that $\Lambda_{\scriptl_\nu}(\bE_\nu)\to \Theta_{\scriptl^0}(\be^0)$
as $\nu\to\infty$.
Then there exists a sequence $t_\nu$ such that
\begin{equation}
|D_{t_\nu}\bE_\nu\symdif \bE^0|\to 0 \text{ as $\nu\to\infty$.}
\end{equation}
\end{lemma}

\begin{proof}
Choose $t_\nu$ so that there exists $\rho>0$ for which
$[-\rho,\rho]^2\subset D_{t_\nu} \bE_\nu$ for all sufficiently large $\nu$.
Assume henceforth that $\nu$ is indeed large,
and replace $\bE_\nu$ by $D_{t_\nu} \bE_\nu$.

Apply the proof of Theorem~\ref{thm:precompact}.
The quantity $\be^0_1$ constructed in that proof must be equal to $\be^0$,
for otherwise a contradiction would be reached immediately using 
Corollary~\ref{cor:strictsublinearity}.
Therefore $\be_{\nu,1}\to 0$.
\end{proof}


While the two notions of strict admissibility on the one hand,
and admissibility on the other, are defined
somewhat differently, they are related. If $\scriptl^0$ satisfies 
the full symmetry hypothesis of Definition~\ref{defn:fullRBLL}, and 
if $(\scriptl^0,\be^0)$ is strictly admissible,
then  $(\scriptl^0,\be^0)$ is admissible
in the sense of Definition~\ref{admissible}.

\begin{lemma}
Let $(\scriptl^0,\be^0)$ be nondegenerate, satisfy the full symmetry
hypothesis \eqref{RBLLsymmetry}, and be strictly admissible.
Then $(\scriptl,\be)$ is admissible whenever $\scriptl$ satisfies 
the partial symmetry hypothesis \eqref{PSH}
and $(\scriptl,\be)$ is sufficiently close to $(\scriptl^0,\be^0)$.
\end{lemma}

\begin{proof}
Consider any sequence $(\scriptl_\nu,\be_\nu)$ that satisfies
the hypotheses and tends to $(\scriptl^0,\be^0)$ as $\nu\to\infty$.
There exists $\be'_\nu\le\be_\nu$
such that $(\scriptl_\nu,\be'_\nu)$ is admissible
and satisfies $\Theta_{\scriptl_\nu}(\be'_\nu)
= \Theta_{\scriptl_\nu}(\be_\nu)$.

In this situation, $\be'_\nu\to\be^0$. 
Indeed, suppose instead that after passing to a subsequence,
$\be'_\nu\to \be'<\be^0$.
Each component of $\be'$ is strictly positive,
since $|\Theta_{\scriptl_\nu}(\be'_\nu)|$ is uniformly minorized by 
a positive quantity, and
is majorized by a constant multiple of the product of the
two smallest components of $\be'_\nu$, uniformly in $\nu$.
By Lemma~\ref{lemma:continuityWRTL}, 
\[
\Theta_{\scriptl^0}(\be) 
= \lim_{\nu\to\infty}\Theta_{\scriptl_\nu}(\be'_\nu)
= \lim_{\nu\to\infty}\Theta_{\scriptl_\nu}(\be_\nu)
= \Theta_{\scriptl^0}(\be^0).\] 

This forces $\be'= \be^0$, that is, $\be'_\nu\to\be^0$.
Indeed, the function $\tilde\be\mapsto\Theta_{\scriptl^0}(\tilde\be)$
is nondecreasing. Since $(\scriptl^0,\be^0)$
is strictly admissible, this function is strictly
increasing in a neighborhood of $\be^0$.
That is, if $\be'<\be''$ and both are close to $\be^0$,
then $\Theta_{\scriptl^0}(\be')<\Theta_{\scriptl^0}(\be'')$.
Therefore if $\be'<\be^0$ then $\Theta_{\scriptl^0}(\be')
<\Theta_{\scriptl^0}(\be^0)$, contradicting the conclusion
of the preceding paragraph.

Let $\bE_\nu$ satisfy $|\bE_\nu|=\be'_\nu$,
$\bE_\nu=\bE_\nu^\dagger$, and $
\Lambda_{\scriptl_\nu}(\bE_\nu)=
\Theta_{\scriptl_\nu}(\be'_\nu)=\Theta_{\scriptl_\nu}(\be_\nu)$.
By replacing $\bE_\nu$ by $D_{t_\nu}(\bE_\nu)$
for appropriate parameters $t_\nu\in(0,\infty)$
we may assume that the associated component sets 
satisfy $|E_{\nu,i}\symdif E^0_{\nu,i}|\to 0$
as $\nu\to\infty$ for each $i\in\scripti$.

It follows that the associated functions $K_i^0,K_{\nu,i}$
satisfy $K_{\nu,i}\to K^0_i$ in $C^0(\reals^2)$ norm.
A consequence of strict admissibility is that
$K_i^0$ is bounded below by a strictly positive quantity
in a neighborhood of the closure of $E_i^0$. 
Therefore $K_{\nu,i}$ is also bounded below by a strictly positive
quantity in a subset of $\reals^2\setminus E_{\nu,i}$
whose Lebesgue measure is also bounded below by a strictly
positive quantity, uniformly in $\nu$.
Write $\be'_\nu = (e'_{\nu,i}: i\in\scripti)$.
Let $r>e'_{\nu,i}$ be close to $e'_{\nu,i}$.
Define $\tilde \be=(\tilde e_k: k\in\scripti)$ by $\tilde e_j=e'_{\nu,j}$
for $j\ne i$ and $\tilde e_i=r$.
Construct $\tilde\bE$ satisfying $|\tilde\bE|=\tilde\be$
with $\tilde E_j=E_{\nu,j}$ for $j\ne i$,
$\tilde E_i\supset E_{\nu,i}$,
and $K_{\nu,i}$ strictly positive on $\tilde E_i\setminus E_{\nu,i}$.
Then 
\[ \Theta_{\scriptl_\nu}(\tilde\be) 
\ge \Lambda_{\scriptl_\nu}(\tilde\bE)
= \int_{\tilde E_i} K_{\nu,i}
> \int_{E_{\nu,i}} K_{\nu,i}
= \Lambda_{\scriptl_\nu}(\bE_\nu)
= \Theta_{\scriptl_\nu}(\be'_\nu).\] 
This holds for every $r>e'_{\nu,i}$, for each $i\in\scripti$.
Since $\be'_\nu<\be_\nu$, this forces $\Theta_{\scriptl_\nu}(\be'_\nu)<
\Theta_{\scriptl_\nu}(\be_\nu)$, which is a contradiction.
Therefore $(\scriptl_\nu,\be_\nu)$ is admissible for every sufficiently large $\nu$.
\end{proof}


\section{Convexity and regularity of maximizers}\label{perturb}

The set of all $4$ tuples $\scriptl=(L_j: j\in\scripti)$
of linear mappings $L_j:\reals^4\to\reals^2$
is a finite-dimensional vector space.
Choose any norm $\norm{\cdot}$ for this space, and fix it for the remainder of the
analysis. This allows us to quantify the difference between $\scriptl$ and $\scriptl^0$
as $\norm{\scriptl-\scriptl^0}$.

We say that a set $E$ is a superlevel set of a function $K\ge 0$ if
there exists $t>0$ such that either
$\big|E\symdif\{x: K(x)>t\}\big|=0$,
or $\big|E\symdif\{x: K(x) \ge t\}\big|=0$.
Let $K$ and $m>0$ be given, and
suppose that
$0<m < |\{x: K(x)>0\}|$,
and that $|\{x: K(x)=t\}|=0$, where $t>0$ is specified by
\[ |\{x: K(x)>t\}|\le m \le |\{x: K(x)\ge t\}.\]
Then a set $E$ maximizes the quantity $\int_{E}K$ 
among all sets of Lebesgue measures $m=|E|$
if and only if $E$ is a superlevel set of $K$. 

Throughout section~\ref{perturb},
$\scriptl^0$ is assumed to be nondegenerate and to satisfy 
the full symmetry hypothesis \eqref{RBLLsymmetry}, and
$(\scriptl^0,\be^0)$ is assumed to be strictly admissible.
It is assumed that
$\scriptl$ is nondegenerate and satisfies 
the partial symmetry hypothesis \eqref{PSH},
that $(\scriptl,\be)$ is admissible,
and that $(\scriptl,\be)$ is a small perturbation of $(\scriptl^0,\be^0)$.

Let $\bE=\{E_j: i\ne j\in\set4\}$ be fixed and satisfy $|E_j|=e_j$.
Let the functions $K_i$ be defined in terms of the sets $E_j$
by \eqref{Kidefn},\eqref{Kiintegral}.
We will prove Theorem~\ref{thm:bootstrap} via
a bootstrapping argument, in which properties of $\{E_i: i\in\set4\}$
are used to deduce information concerning $\{K_i: i\in\set4\}$, leading in turn to stronger
properties of $\{E_i: i\in\set4\}$. 
Several iterations lead to the desired conclusions.
The most involved step is Lemma~\ref{lemma:lipschitz} below,
which establishes smallness of $K_i-K_i^0$ in the Lipschitz norm,
leading to the conclusion that $E_i$ is a Lipschitz domain
whose boundary is close to that of $E_i^0$ in the Lipschitz sense.

\medskip
As was proved in \cite{christoneill},
since $(\scriptl^0,\be^0)$ is strictly admissible,
if $\bE^0$ satisfies $|\bE^0|=\be^0$
then $\bE^0$ maximizes $\Lambda^0$ if and only if 
$\bE^0$ is a $4$--tuple of homothetic ellipses, 
whose centers form a $4$--tuple that belongs to the orbit
of $(0,0,0,0)$ under the translation symmetry group.
Among these tuples are those with sides parallel to the axes and with each $E_j^0$
centered at $0\in\reals^2$. 
These special maximizers of $\Lambda^0$ naturally play a distinguished role in the analysis. 

According to Theorem~\ref{thm:precompact}, 
if $(\scriptl^0,\be)$ is strictly admissible,
if $(\scriptl,\be)$ is admissible and 
$\scriptl$ is sufficiently close to $\scriptl^0$,
and if $\bE$ is a maximizer for $\Lambda_\scriptl$ satisfying $\bE=\bEdagger$ and $|\bE|=\be$, 
then there exists $t$ such that $D_t\bE$ is close to a special maximizer $\bE^0$
of $\Lambda^0$, in the sense that
\[|D_tE_j\symdif E_j^0|<\eps \text{ for each $j\in\set4$,} \]
where $\eps\to 0$ as $\norm{\scriptl-\scriptl^0}\to 0$. 
We will prove the conclusions of Theorem~\ref{thm:bootstrap} for a dilate $D_t$ of $\bE$ chosen so that for each $j\in\set4$, $|D_tE_j\Delta E_j^0|<\eps$ where $E_j^0$ is a ball centered at the origin. 
Assume henceforth that $\bE^0$ is a $4$--tuple of balls centered at the origin.

By $o_\delta(1)$ we mean a quantity that depends on $\scriptl^0$
and on $\be$, but tends to $0$ as $\delta\to 0$
provided that $(\scriptl^0,\be)$ remains fixed,
or more generally, remains inside a compact subset of the set of
all $(\scriptl^0,\be)$ that satisfy our hypotheses.

\begin{lemma}
Let $\scriptl^0$ be nondegenerate and satisfy the full symmetry hypothesis
of Definition~\ref{defn:fullRBLL}.
Let $(\scriptl^0,\be)$ be strictly admissible and generic.
Let $\bE^0$ be the $4$--tuple of balls in $\reals^2$ centered at the origin
satisfying $|\bE^0|=\be$. For each $i\in\set4$ there exists a neighborhood of
$\partial E_i^0$ in which $K_i^0$ is $C^\infty$ and has nowhere vanishing gradient.
\end{lemma}

\begin{proof} 
Since $E_j^0$ are balls centered at the origin for all $j\ne i$,
and since the diagonal action of $O(2)$ on $(\reals^2)^4$ defines a symmetry of $\Lambda_{\scriptl^0}$,
$K_i^0$ is a radially symmetric function. Thus it suffices to analyze its gradient at the unique
point in $\partial E_i^0$ of the form $(\bar u,0)$ with $\bar u>0$.

Consider $u$ in a small neighborhood of $\bar u$.
$K_i^0(u,0)$ is the Lebesgue measure of 
\[ \underset{i\ne j\in\set4}\bigcap\, \tilde E_j^0(u,0)
= \underset{i\ne j\in\set4}\bigcap\, (\tilde E_j^0 + (\rho_j u,0))\]
where the three real coefficients $\rho_j$ 
are determined by the mappings $\ell_{j,i}$ defined in \eqref{Kiintegral}
and are pairwise distinct, and $\tilde E_j^0=\tilde E_j^0(0,0)$
is a closed ball centered at the origin. Write $\{j\in\set4: j\ne i\}$
as $\{j,k,l\}$, with the indices labeled so that $\rho_j<\rho_l<\rho_k$.
By making a $u$--dependent translation change of variables in the integral
over $\reals^2$ that defined $K_i^0(u,0)$ we may reduce to the case
in which $\rho_l=0$. Then $\tilde E_l^0(u,0)=\tilde E_l^0$ is a closed
ball centered at the origin in $\reals^2$.

According to the strict admissibility and genericity
hypotheses, either $\tilde E_j^0(\bar u,0)\cap\tilde E_k^0(\bar u,0)$ is contained in
the interior of $\tilde E_l^0$,
or the threefold intersection of these three balls is a convex domain bounded
by circular arcs of positive lengths that meet transversely, with $2$ of these
arcs being subarcs of $\tilde E_n^0(\bar u,0)$ for each $n\in\{j,k,l\}$.
In either case, it is an elementary consequence of the inequality
$\rho_j<0=\rho_l<\rho_k$ that in a neighborhood of $\bar u$, the volume of the threefold intersection
is a $C^\infty$ function of $u$ with strictly negative derivative.\footnote{See also \eqref{C1formula} below.}
\end{proof}

The tuples $\bE=(E_i: i\in\scripti)$ depend on $(\scriptl,\be)$,
but this dependence is not indicated in our notations.

\begin{lemma}
Let $\delta_0>0$ be a sufficiently small constant,
depending only on $\scriptl^0,\be$.
Let $\norm{\scriptl-\scriptl^0}\le\delta\le\delta_0$.
For each $i\in\set4$, $E_i$ is a bounded set whose
diameter is bounded above, uniformly in $\scriptl$.
The Hausdorff distance from $\partial E_i$ to $\partial E_i^0$
is $o_\delta(1)$.
\end{lemma}

\begin{proof}
It follows directly from \eqref{Kiintegral} and the nondegeneracy
hypothesis that $K_i$ is continuous,
and that $K_i(u)\to 0$  as $|u|\to\infty$. The same holds for $K_i^0$.
Moreover, $\norm{K_i}_{C^0} \le C |E_k|\cdot|E_l|$
for any $k\ne l\in \scripti\setminus\{i\}$,
and
\[ \norm{K_i-K_i^0}_{C^0} \le C\max_{j\ne i} |E_j\symdif E_j^0|\]
where $C<\infty$ depends on $\scriptl^0,\be$
and on an upper bound for $\delta$.
Therefore $\norm{K_i-K_i^0}_{C^0} \le o_\delta(1)$.

The strict admissibility hypothesis implies that
each set $E_i^0$ is a superlevel set $\{x: K_i^0(x)\ge t_i>0\}$ of $K_i^0$,
and $\nabla K_i^0$ vanishes nowhere on the boundary of $E_i^0$.
Since $\norm{K_i-K_i^0}_{C^0}$ is small, and
since $\bE$ is a maximizing tuple, 
for each $i\in\set4$,
$E_i\subset\{u: K_i(u)>t_i-o_\delta(1)\}$ provided that
$\delta$ is sufficiently small.
Therefore the diameters of the sets $E_i$ are majorized
by an acceptable constant.

The nonvanishing of $\nabla K_i^0$ in a neighborhood of $\partial E_i^0$
and the smallness of $\norm{K_i-K_i^0}_{C^0}$ together imply 
smallness of the Hausdorff distance from the boundary of $E_i$
to the boundary of $E_i^0$. 
\end{proof}

\begin{lemma} \label{lemma:lipschitz}
If $\delta>0$ is sufficiently small
then each function $K_i$ is Lipschitz continuous,
with Lipschitz constants uniformly bounded above, 
for all $\scriptl$ satisfying $\norm{\scriptl-\scriptl^0}\le\delta$. 
\end{lemma}

For any $i\in\set4$, for each $j\ne i$ define $\tilde E_j$ to be the
inverse image of $E_j$ under the mapping $v\mapsto \ell_{j,i}(0,v)$.
Define $\tilde E_j^0$ in the corresponding way, in terms of $E_j^0$
and $\scriptl^0$.
$\tilde E_j$ could more properly be denoted by $\tilde E_{j,i}$,
but the simplified notation will be sufficiently unambiguous for our purpose.

\begin{proof}
\begin{align*}
|K_i(u)-K_i(u')|
&\le  \int_{\reals^2} \big|
\prod_{j\ne i} \one_{E_j}(\ell_{j,i}(u,v))
- \prod_{j\ne i} \one_{E_j}(\ell_{j,i}(u',v))
\big|
\,dv
\\ &
=  \int_{\reals^2} \big|
\prod_{j\ne i} \one_{E_j}(\ell_{j,i}((u-u'),v))
- \prod_{j\ne i} \one_{E_j}(\ell_{j,i}(0,v))
\big|
\,dv.
\end{align*}
As a function of $v\in\reals^2$,
$\one_{E_j}(\ell_{j,i}(u-u',v))$
and
$\one_{E_j}(\ell_{j,i}(0,v))$
are translates of $\tilde E_j$.
Moreover, they are translates by quantities whose difference
is a linear transformation of $u-u'$.
Therefore it suffices to verify that $|(E_j+w)\symdif E_j| = O(|w|)$
for $w\in\reals^2$.
Moreover, it suffices to verify this merely for $w\in\reals\times\{0\}$,
and for $w\in \{0\}\times\reals^2$.
Since the horizontal and vertical coordinates can be freely interchanged
in this theory, it suffices to treat $w=(z,0)\in\reals\times\{0\}$.
The bound 
$|(E_j+w)\symdif E_j| = O(|w|)$
is an immediate consequence of two properties of $E_j$:
the intersection of $E_j$ with any horizontal line in $\reals^2$
is an interval, and the diameter of $E_j$ is bounded above by an acceptable constant.
\end{proof}

\begin{lemma} \label{lemma:littleLip}
For each $i\in\set4$,
\[ \big|\, (K_i-K_i^0)(u) - (K_i-K_i^0)(u')\,\big| 
\le o_\delta(1)\cdot |u-u'| + O(|u-u'|^2)\]
for all $u,u'$ sufficiently close to $\partial(E_i^0)$.
\end{lemma}

\begin{proof}
Fix any index $i\in\scripti$.
For $u\in\reals^2$
define \[\Omega_i(u)
=\{v\in\reals^2: \ell_{j,i}(u,v)\in E_j\ \forall j\ne i\}
= \bigcap_{j\ne i} \tilde E_j(u).  \]
Thus $K_i(u) = |\Omega_i(u)|$.
Likewise define $\Omega_i^0(u)$ in terms of $\scriptl^0$ and the sets $E_j^0,\tilde E_j^0$.
For each $u$ in a neighborhood of $\partial E_i^0$,
$\Omega_i^0(u)$ is a bounded connected set,
whose boundary is a union of two or four arcs of circles. 
The genericity hypothesis guarantees that
these arcs meet transversely at any points of intersection,
and that only two arcs meet at any such point.
These are subarcs of translates of the boundaries of the balls
$\tilde E_j^0$, respectively.
The intersection of $\Omega_i(u)$ with any horizontal or vertical
line is empty, or is an interval.

Each set $\tilde E_j$ is invariant with respect to reflection about both
the horizontal and vertical axes. In the first quadrant, the boundary
of $\tilde E_j$ is a rectifiable curve that can be parametrized by arclength
as $s\mapsto (x(s),y(s))$ with 
\begin{equation} \label{arcparameters}
\dot x(s)\ge 0\ \text{ and } \  \dot y(s)\le 0.
\end{equation}
This curve is contained in an $o_\delta(1)$--neighborhood of the boundary of $\tilde E_j^0$.
The sets $\tilde E_j(u),\tilde E_j^0(\bar u)$
are translates of $\tilde E_j,\tilde E_j^0$, respectively.

Let $\delta>0$ be small, and let $\scriptl$ satisfy
$\norm{\scriptl-\scriptl^0}\le\delta$.
Consider any $i\in\set4$, any $\bar u$ in an $o_\delta(1)$--neighborhood
of $\partial \tilde E_j^0$, and any $u$ near $\bar u$. 
Choose a collection of two or four disks $\{Q_\alpha\}$ in $\reals^2$ whose radii 
tend to zero slowly as $\delta\to 0$,
centered at the two or four intersection points of the arcs comprising 
the boundary of $\Omega_i^0(\bar u)$. 
The boundary if $\Omega_i(u)$ then consists of 
two or four arcs of circles in the boundaries of the disks $Q_\alpha$,
together with two or four rectifiable curves, each of which
has Hausdorff distance $o_\delta(1)$ to a subarc of one of the circular
arcs comprising the boundary of $\Omega_i^0(\bar u)$,
is a translate of a subarc of the boundary of $\tilde E_j$
for some $j\in\scripti$, 
and has monotonicity properties thereby inherited from \eqref{arcparameters}.

For $u$ sufficiently close to $\bar u$,
the boundaries $\partial\tilde E_k(u)$ enjoy the following two properties.
For each $k\in\set4$ and each $\alpha$,
if $\partial\tilde E_k(\bar u)$ meets $\partial Q_\alpha$
then these intersect at a single point $z(k,\alpha,\bar u)$.
There exists a subarc $\Gamma(u)$ of $\partial \tilde E_k(u)$
of arclength $O(|u-\bar u|)$
such that the portion of $\partial\tilde E_k(u)$ not lying
in $\Gamma(u)$ lies at
distance $\ge C_0 |u-\bar u|$ from the boundary of $Q_\alpha$.
Moreover, for all $u$ sufficiently close to $\bar u$,
$\partial\tilde E_k(u)$ meets $\partial Q_\alpha$
at a single point, and this point lies within distance $O(|u-\bar u|)$
of $z(k,\alpha,\bar u)$.
These properties are consequences of the monotonicity properties \eqref{arcparameters}
of the boundaries of $\tilde E_j$ and the assumption that $u$ takes the form $(u_1,0)$.

For $u\in\reals^2$ near $\bar u$, for sufficiently small $\delta$,
\begin{align*} K_i(u) 
= \int_{\Omega_i(u)} dv_1 \wedge dv_2
= \int_{\Omega_i(u)\setminus\cup_\alpha Q_\alpha} d(v_1  \,dv_2)
+ \sum_\alpha \int_{Q_\alpha\cap \Omega_i(u)}  \,dv_1\wedge dv_2.
\end{align*}
The same reasoning as in the proof of Lemma~\ref{lemma:lipschitz}
shows that each term 
$\int_{Q_\alpha\cap \Omega_i(u)}  \,dv_1\wedge dv_2$
defines a locally Lipschitz function of $u$, whose Lipschitz norm is $o_\delta(1)$ 
because the Lebesgue measure of $Q_\alpha$ is $o_\delta(1)$.
$K_i^0(u)$ can be analyzed in the same way, producing a corresponding
term that is also Lipschitz with norm $o_\delta(1)$.

The main term for $K_i(u)$, 
$\int_{\Omega_i(u)\setminus \cup_\alpha Q_\alpha} d(v_1  \,dv_2)$,
can be rewritten via Stokes' theorem. What results is a sum 
of integrals of the one-form $\omega = v_1\,dv_2$
over finitely many rectifiable arcs
$\gamma_\beta(u)$.
Each of these arcs is either a subarc of the boundary of a single $\tilde E_k(u)$,
or is a subarc of the boundary of some $Q_\alpha$.
Label these arcs so that for each index $\beta$,
$\gamma_\beta(u)$ is close in the Hausdorff metric to each of $\gamma_\beta(\bar u)$,
$\gamma_\beta^0(u)$, and $\gamma_\beta^0(\bar u)$,
provided that $\delta$ and $|u-\bar u|$ are sufficiently small.

Consider the contribution of an arbitrary $\gamma(u)=\gamma_\beta(u)$ of the former type.
Its contribution is $\int_{\gamma(u)} v_1\,dv_2$.
We wish to compare this quantity to $\int_{\gamma(\bar u)}v_1\,dv_2$.
$\gamma(u)$ is a subarc of the full boundary $\partial\tilde E_k(u)$ for some index $k\in\set4\setminus\{i\}$.
This full boundary may be expressed as a translate
\[\partial\tilde E_k(u) = \partial \tilde E_k(\bar u) + \ell_{k,i}^\sharp(u-\bar u)\]
for a certain linear mapping $\ell_{k,i}^\sharp: \reals^2\to\reals^2$,
which differs by $o_\delta(1)$ from the corresponding mapping 
$\ell_{k,i}^\sharp$ associated to $\scriptl^0$.

Denote the first component of the $\reals^2$--valued linear map $\ell_{k,i}^\sharp$ by $\tilde\ell_{k,i}$.
The contribution of $\gamma(u)$ is
\begin{align*}
\int_{\gamma(u)} v_1\,dv_2
&= \int_{\gamma(u)-\ell_{k,i}^\sharp(u-\bar u)} (v_1 + \tilde\ell_{k,i}(u-\bar u)) \,dv_2
\\&
= \int_{\gamma(\bar u)} (v_1 + \tilde\ell_{k,i}(u-\bar u)) \,dv_2
+\scriptr(u,\bar u)
\end{align*}
where the remainder $\scriptr(u,\bar u)$ is expressed as an integral over
the symmetric difference between $\gamma(\bar u)$ and $\gamma(u,\bar u)
= \gamma(u)-\ell_{k,i}^\sharp(u-\bar u)$
of an integrand of the form $v_1\,dv_2 + O(|u-\bar u|)$.

Both $\gamma(\bar u)$ and $\gamma(\bar u,u)$ are subarcs of $\partial \tilde E_k(\bar u)$,
so their symmetric difference is a union of two or fewer rectifiable arcs. 
We claim that
each of these  two or fewer arcs has diameter $O(|u-\bar u|)$.
Indeed, if $z\in\partial \tilde E_k(\bar u)$
lies at distance $\ge C_0|u-\bar u|$ from the boundary
of $Q_\alpha$ then $z+\ell_{k,i}^\sharp(u-\bar u)\in\partial\tilde E_k(u)$ 
shares this property with $C_0$ replaced by $C_0/2$, and conversely,
provided that the constant $C_0$ is chosen to be sufficiently large.
Thus the portion of $\gamma(\bar u)\symdif \gamma(u,\bar u)$
that lies within distance $o_\delta(1)$ of the boundary of $Q_\alpha$
lies entirely within distance $O(|u-\bar u|)$ of $\partial Q_\alpha$.
By the key property of $\partial\tilde E_k(\bar u)$ noted above,
this establishes the claim.

The corresponding quantity associated to $K_i(\bar u)$ is simply $\int_{\gamma(\bar u)} v_1\,dv_2$.
Subtracting this from $\int_{\gamma(u)} v_1\,dv_2$ yields
\[ -\tilde\ell_{k,i}(u-\bar u)\, \int_{\gamma(\bar u)} \,dv_2 + \scriptr(u,\bar u).  \]
The factor $\int_{\gamma(\bar u)} \,dv_2$ is independent of $u$.
Moreover, it is the integral of an exact one-form over the curve $\gamma(\bar u)$,
and consequently depends only on the two endpoints of this curve.

Analyzing $K_i^0(u)-K_i^0(\bar u)$ in the same way results in a corresponding term,
which depends in the same way on the two endpoints of the corresponding
elliptical arc $\gamma(\bar u)$. 
Because the Hausdorff distance between $\partial E_k$  and $\partial E_k^0$
is $o_\delta(1)$, the difference between these two contributions
is therefore $O(|u-\bar u|)\cdot o_\delta(1)$. 

To complete discussion of the contribution of $\gamma$,
it remains to analyze $\scriptr(u,\bar u) - \scriptr^0(u,\bar u)$.
Consider the two rectifiable curves that comprise
the symmetric difference between $\gamma(\bar u)$ and $\gamma(\bar u,u)$.
On each, the function $v_1$ may be expressed as a constant plus $O(|u-\bar u|)$.
Terms that are $O(|u-\bar u|)$ produce contributions that are $O(|u-\bar u|^2)$
since the integrals here are taken over curves whose lengths are $O(|u-\bar u|)$.
As above, the constant terms give rise to integrands which are constant multiples of $dv_2$.
Subtracting the corresponding terms for $K_i^0$ and exploiting exactness of $dv_2$,
we conclude that 
\[ |\scriptr(u,\bar u)-\scriptr^0(u,\bar u)| \le O(|u-\bar u|)o_\delta(1) + O(|u-\bar u|^2).\]

The analysis of subarcs of the boundaries of the disks $Q_\alpha$
is very slightly simpler, since these arcs are not translated.
Their contributions are entirely of the type of the remainders $\scriptr(u,\bar u)$.
The same analysis as carried out for $\scriptr(u,\bar u)$ above applies to them.
\end{proof}

\begin{corollary} \label{cor:Lip}
For each $i\in\set4$, $E_i$ is a Lipschitz domain, 
uniformly for all $\scriptl$ sufficiently close to $\scriptl^0$.
\end{corollary}

\begin{proof}
In a neighborhood of the boundary of $E_i^0$,
$K_i^0$ is a $C^\infty$ function with nowhere vanishing gradient.
Since $\norm{K_i-K_i^0}_{\rm Lip}\le o_\delta(1)$,
if $\delta$ is sufficiently small
then $E_i=\{u: K_i(u)\ge t_i=t_i(\scriptl,\be)\}$
is a Lipschitz domain whose boundary lies in an $o_\delta(1)$--neighborhood
of the boundary of $E_i^0$.
\end{proof}

\begin{proof}[Conclusion of proof of Theorem~\ref{thm:bootstrap}]
Continuing the discussion in the proof of Corollary~\ref{cor:Lip},
for any $u\in\partial E_i$,
for any $u'$ sufficiently near $u$, $u-u'$
lies in a cone of aperture $o_\delta(1)$
centered around a vector that is tangent to $\partial E_i^0$
at a point whose distance to $u$ is $o_\delta(1)$.
Therefore $\partial\Omega_i(u)$ is a Lipschitz domain
whose boundary consists of finitely many Lipschitz arcs $\gamma_\beta(u)$, with each
$\gamma_\beta(u)$ being a subarc of the boundary of a translate of $\tilde E_k$
for some $i\ne k\in\set4$. Denote the endpoints of $\gamma_\beta(u)$ 
by $\bx_{\beta}(u),\bx'_\beta(u)$.
At any intersection of these arcs,
only two arcs meet, and any such intersection is transverse
in the sense that tangent cones are separated by a positive angle.

Inserting this information into
the proof of Lemma~\ref{lemma:littleLip}, the disks $Q_\alpha$ can now be dispensed with,
yielding the representation
\begin{equation} K_i(u) = \sum_\beta \int_{\gamma_\beta(u)} v_1\,dv_2\end{equation}
for all $u$ in some neighborhood of $\partial E_i^0$.

From this and the fact that $\gamma_\beta(u)$
is a subarc of a translate of the Lipschitz boundary of $\tilde E_k(\bar u)$
by a linear function of $u-\bar u$
we deduce that $K_i\in C^1$ in a neighborhood of $\partial E_i$ with 
$\nabla K_i$ expressible as
\begin{equation} \label{C1formula}
\nabla K_i(u) 
= \sum_\beta\int_{\gamma_\beta(u)} w_{k(\beta)}(u-\bar u) \,dv_1
= \sum_\beta w_{k(\beta)}(u-\bar u)\, (y_\beta(u)-y'_\beta(u))
\end{equation}
where $y_\beta,y'_\beta\in\reals$ are the second coordinates of 
$\bx_\beta,\bx'_\beta\in\reals^2$, respectively,
and $w_k(u-\bar u)$ are $\reals^2$--valued linear functions of $u-\bar u$ determined by $\scriptl$
and the indices $i,k$, mediated by the function $(v_1,v_2)\mapsto v_1$.
Therefore $|w_{k(\beta)}(u-\bar u)|=O(|u-\bar u|)$.
The representation \eqref{C1formula} holds for $C^1$ boundaries, and follows for Lipschitz 
boundaries from the $C^1$ case by a limiting argument.

Therefore the superlevel set $E_i$ of $K_i$ is also a $C^1$ domain. 
It follows that the endpoints $\bx_\beta(u),\bx'_\beta(u)$ are themselves $C^1$
functions of $u$, since they are translates by affine functions of $u$ 
of transversely intersecting $C^1$ arcs. 
Moreover, $\bx_\beta(u)-\bx_\beta^0(u)$ and its gradient with respect to $u$
are uniformly $o_\delta(1)$.
Likewise for $\bx'_\beta(u)$.

Inserting this information into \eqref{C1formula}, we conclude that $K_i\in C^2$.
Since 
$\bx_\beta(u)-\bx_\beta^0(u)$ and its gradient with respect to $u$
are uniformly $o_\delta(1)$
and likewise for $\bx'_\beta(u)$,
it follows moreover that $\norm{K_i-K_i^0}_{C^2}=0_\delta(1)$
in a neighborhood of $\partial E_i^0$.
Therefore $E_i$ is strongly convex.

This reasoning can be iterated to conclude that $K_i\in C^\infty$.
\end{proof}


\section{Nonexistence of maximizers} \label{section:nonexistence}

In this section we discuss a family of data that do not satisfy
the partial symmetry hypothesis \eqref{PSH}. 
In exceptional cases these data are reducible via simple skew-shift changes of variables
to data that satisfy the full symmetry hypothesis.
We show that for all other data of this special type, maximizers $\bE$ fail to exist.
Thus our partial symmetry hypothesis \eqref{PSH} is less artificial than it may appear to be.

Let $L_j^0(\bx,\by) = (L_{j,1}^0(\bx),L_{j,2}^0(\by))$
with $L_{j,1}^0=L_{j,2}^0$,
so that $\scriptl^0=(L_j^0: j\in\scripti)$
satisfies the full symmetry hypothesis.
Suppose that $\scriptl^0$ is nondegenerate.

Let 
$\ell_j:\reals^2\to\reals^1$ be linear, and consider
$\scriptl=(L_i: i\in\scripti)$ with
\begin{equation} L_j(\bx,\by) = (L_{j,1}(\bx),L_{j,2}(\bx,\by))\end{equation}
of the form
\begin{equation} L_{j,1}=L_{j,1}^0 \text{ and } 
L_{j,2}(\bx,\by)=L_{j,2}^0(\by) + \ell_j(\bx).\end{equation}
Define
$\ell:\reals^2\to\reals^4$ by $\ell(\bx) = (\ell_i(\bx): i\in\scripti)$ and $L_2^0:\R^2\to\R^4$ by $(L_{j,2}^0(\by):y\in\set4)$.

\begin{proposition}[Nonexistence of maximizers] \label{prop:nonexistence}
Let $\scriptl^0$ be nondegenerate.
Let $(\scriptl^0,\be)$ be strictly admissible.
If the range of $\ell$ is not contained in the range of $L^0_2$,
then there exists no $4$--tuple $\bE$ satisfying
$|\bE|=\be$ and
$\Lambda_\scriptl(\bE)=\Theta_\scriptl(\be)$.
\end{proposition}

Thus maximizers can fail to exist for arbitrarily small
perturbations $\scriptl$ of $\sl(d)$--invariant data $\scriptl^0$.

We believe that this remains true if the full symmetry hypothesis is relaxed to \eqref{PSH},
but the proof below utilizes a property of maximizers that has as yet been
established under only the full symmetry hypothesis, or (as a corollary, using Theorem~\ref{thm:bootstrap})
for partially symmetric data that are sufficiently small perturbations of fully symmetric data. 

The proof of Proposition~\ref{prop:nonexistence} uses partial symmetrization.
For $E\subset\reals^2$, define $E^\ddagger$ to be the vertical
symmetrization of $E$. Thus for each $x\in\reals$, define $E_x=\{y\in\reals^1:
(x,y)\in E\}$, define $E_x^\ddagger=[-\tfrac12 |E_x|,\tfrac12|E_x|]$
if $|E_x|>0$ and $E_x^\ddagger=\emptyset$ if $|E_x|=0$,
and define $E^\ddagger = \{(x,y)\in\reals^2: y\in E_x^\ddagger\}$.

As above, for $t\in(0,\infty)$ define
$D_t(E)=\{(tx,t^{-1}y): (x,y)\in E\}$. 
Define $D_t(\bE) = (D_t(E_j): j\in\scripti)$. 

\begin{lemma}
For any $\bE$,
\begin{equation} \label{compare} 
\Lambda_{\scriptl}(\bE) \le \Lambda_{\scriptl^0}(\bE^\ddagger).\end{equation}
\end{lemma}

\begin{proof}
Consider
\begin{align*}
\Lambda_{\scriptl}(\bE) 
&= \int_{\reals^2} \int_{\reals^2} \prod_{j\in\scripti}
\one_{E_j}(L_j(\bx,\by))\,d\by\,d\bx
\\&
= \int_{\reals^2} \int_{\reals^2} \prod_{j\in\scripti}
\one_{E_{j,{L_{j,1}(\bx)}}}
(L_{j,2}^0(\by) + \ell_j(\bx))\,d\by\,d\bx
\end{align*}
where $E_{j,z} = \{y: (z,y)\in E_j\}$.
For each $\bx\in\reals^2$,
apply the symmetrization inequality of Rogers-Brascamp-Lieb-Luttinger
to the inner integral. No symmetry hypothesis comes into play,
since each set $E_{j,z}$ is a subset of $\reals^1$.
Therefore
\begin{align*}
\Lambda_{\scriptl}(\bE) 
\le \int_{\reals^2} \int_{\reals^2} \prod_{j\in\scripti}
\one_{ E_{j,{L_{j,1}(\bx)}}^\star} 
(L_{j,2}^0(\by))\,d\by\,d\bx
\end{align*}
where 
$E_{j,z}^\star \subset\reals^1$ 
is the usual symmetrization of $E_{j,z}$; it is the closed interval centered
at $0\in\reals^1$ whose Lebesgue measure equals that of $E_{j,z}$
if this measure is strictly positive, and is empty otherwise.
The right-hand side of this last inequality is equal to
$\Lambda_{\scriptl^0}(\bE^\ddagger)$.
\end{proof}

This proof of \eqref{compare} yields supplementary information that is essential to our purpose:
If $\bE=\bE^\ddagger$, then
$\Lambda_\scriptl(\bE)=\Lambda_{\scriptl^0}(\bE)$
if and only if for almost every $\bx\in\reals^2$,
\begin{equation}
\int_{\reals^2} \prod_{j\in\scripti}
\one_{E_{j,{L_{j,1}(\bx)}}} (L_{j,2}^0(\by) + \ell_j(\bx))\,d\by
=
\int_{\reals^2} \prod_{j\in\scripti} \one_{E_{j,{L_{j,1}(\bx)}}} (L_{j,2}^0(\by))\,d\by.
\end{equation}

Upon taking the supremum over all $\bE$ satisfying $|\bE|=\be$, 
we conclude from \eqref{compare} that
$\Theta_{\scriptl}(\be) \le \Theta_{\scriptl^0}(\be)$  for any $\be$.
More is true:
\begin{proposition}
Let $\ell:\reals^2\to\reals^4$ be an arbitrary linear map.
Let $\scriptl,\scriptl^0$ be as described above.
For any $\be\in(0,\infty)^4$,
\begin{equation} \label{alsomaxzero}
\Theta_{\scriptl}(\be) = \Theta_{\scriptl^0}(\be). \end{equation}
\end{proposition}

\begin{proof}
By \eqref{compare},
it suffices to show that for any $4$--tuple satisfying $\bE=\bE^\ddagger$,
\begin{equation} \Lambda_{\scriptl}(D_t\bE) \to \Lambda_{\scriptl^0}(\bE) \ \text{ as $t\to 0$.} \end{equation}
To evaluate this limit write
\begin{align*}
\Lambda_{\scriptl}(D_t\bE) 
&= \int_{\reals^2} \int_{\reals^2} \prod_{j\in\scripti}
\one_{D_t E_j}\big(L^0_{j,1}(\bx),L^0_{j,2}(\by)+\ell_j(\bx)\big)\,d\by\,d\bx
\\
&= \int_{\reals^2} \int_{\reals^2} \prod_{j\in\scripti}
\one_{E_j}\big(t^{-1} L^0_{j,1}(\bx),t L^0_{j,2}(\by)+ t \ell_j(\bx)\big)\,d\by\,d\bx
\\
&= \int_{\reals^2} \int_{\reals^2} \prod_{j\in\scripti}
\one_{E_j}\big(L^0_{j,1}(\bx), L^0_{j,2}(\by)+ t^2 \ell_j(\bx)\big)\,d\by\,d\bx
\end{align*}
by substituting $\bx = t\,\bu$ and $\by = t^{-1}\,\bv$ and then
replacing $(\bu,\bv)$ by $(\bx,\by)$ to obtain the last line.

For any interval $I\subset\reals$ centered at the origin
$|I\symdif(I+t)|\le |t|$.
The conclusion \eqref{alsomaxzero} 
follows by applying this bound together  with routine majorizations
to the expression for $\Lambda_\scriptl(D_t\bE)$
in the final line of the chain of identities in the preceding paragraph.
\end{proof}

\begin{proof}[Proof of Proposition~\ref{prop:nonexistence}]
Suppose that
$\bE$ were a maximizer for $\Lambda_\scriptl$. 
Then $\bE^\ddagger$ is also a maximizer for $\Lambda_\scriptl$, by \eqref{compare}. 
Moreover, by \eqref{alsomaxzero}, 
$\bE^\ddagger$ is a maximizer for $\Lambda_{\scriptl^0}$. 
So consider any common maximizer for $\Lambda_\scriptl$ and for $\Lambda_{\scriptl^0}$
that satisfies $\bE = \bE^\ddagger$.

Arbitrary maximizers for $\Lambda_{\scriptl^0}$
have been characterized for nondegenerate strictly admissible $(\scriptl^0,\be)$ 
\cite{christoneill}.
They have the property that
there exists $\eta>0$, depending on $\bE$, such that
for all $\bx\in\reals^2$ satisfying $|\bx|\le\eta$,
the $4$--tuple 
$(|E_{i,{L_{i,1}(\bx)}}|: i\in\scripti)$ 
of measures of associated one-dimensional sets is strictly admissible
for the lower-dimensional form 
\begin{equation} \label{lowerDform}
\int_{\reals^2} \prod_{j\in\scripti} \one_{F_j}(L_{j,2}^0(\by))\,d\by,
\end{equation}
where $(F_j: j\in\scripti)$ represents a $4$--tuple of subsets of $\reals^1$. 

Maximizers of \eqref{lowerDform} have been characterized \cite{christBLLeq}
under hypotheses of nondegeneracy, strict admissibility, and genericity.
These hypotheses are satisfied by $(L^0_{j,2}: j\in\scripti)$ and $\be$. 
We conclude that for any $\bx$ for which
$(|E_{i,{L_{i,1}(\bx)}}|: i\in\scripti)$ is a strictly admissible $4$--tuple,
the vector $\ell(\bx)=(\ell_j(\bx): j\in\scripti) \in\reals^4$ 
takes the form $(L^0_{j,2}(\bu): j\in\scripti)$ for some $\bu\in\reals^2$;
that is, $\ell(\bx)$ belongs to the range of $L^0_2$.
Since $\ell,L^0_2$ are linear mappings, 
the range of $\ell$ is contained in the range of $L^0_2$, as claimed.

\medskip
Conversely, if the range of $\ell=(\ell_i: i\in\scripti)$ 
is contained in the range of $L^0_2 = (L_{2,i}^0: i\in\scripti)$,
then 
the theory developed in this paper for data $\scriptl$ 
satisfying the partial symmetry hypothesis \eqref{PSH} can be applied
to $\Lambda_{\scriptl}$ after a simple change of variables. It is not necessary
to assume that $\scriptl^0$ satisfies the full symmetry hypothesis \eqref{RBLLsymmetry}. Indeed
the hypothesis of inclusion of the ranges implies that there exists a
linear mapping $h:\reals^2\to\reals^2$ satisfying
$\ell_i = L^0_{2,i}\circ h$ for every $i\in\scripti$.
Thus
\begin{align*}
\Lambda_{\scriptl}(\bE) 
& = \int_{\reals^2} \int_{\reals^2} \prod_{j\in\scripti}
\one_{E_{j,{L_{j,1}(\bx)}}}
(L_{j,2}^0(\by + h(\bx)))\,d\by\,d\bx.
\end{align*}
The linear change of variables $(\bx,\by)\mapsto (\bx,\by+h(\bx))$
in $\reals^4$ transforms this double integral to
\begin{align*}
\int_{\reals^2} \int_{\reals^2} \prod_{j\in\scripti}
\one_{E_{j,{L_{j,1}(\bx)}}}
(L_{j,2}^0(\by))\,d\by\,d\bx
= \Lambda_{\scriptl^0}(\bE).
\end{align*}
\end{proof}


\begin{thebibliography}{20}

\bibitem{barthe}
F.~Barthe, {\em On a reverse form of the Brascamp-Lieb inequality}, Inv.~Math. 134 (1998), 335--–361

\bibitem{BBFL}
J.~Bennett, N.~Bez, T.~Flock, and S.~Lee,
{\em Stability of the Brascamp-Lieb constant and applications}, 
American Journal of Mathematics, Volume 140, Number 2, April 2018, pp. 543--569.

\bibitem{BBCF}
J.~Bennett, N.~Bez, M.~Cowling, and T.~Flock, 
{\em Behaviour of the Brascamp--Lieb constant}, 
arXiv:1605.08603 math.CA 

\bibitem{BCCT1}
J.~Bennett, A.~Carbery, M.~Christ, and T.~Tao,
{\em The Brascamp-Lieb inequalities: finiteness, structure and extremals}, 
Geom.~Funct.~Anal.~17 (2008), no. 5, 1343--1415

\bibitem{BCCT2}
\bysame,
{\em Finite bounds for H\"older-Brascamp-Lieb multilinear inequalities}, 
Math.~Res.~Lett.~17 (2010), no.~4, 647--666



\bibitem{brascamplieb}
H.~J.~Brascamp and E.~H.~Lieb,
{\em Best constants in Young's inequality, its converse, and its generalization 
to more than three functions}, Advances in Math. 20 (1976), no. 2, 151–173

\bibitem{BLL}
H.~J.~Brascamp, E.~H.~Lieb, and J.~M.~Luttinger,
{\em A general rearrangement inequality for multiple integrals}, 
J.\ Functional Analysis 17 (1974), 227--237

\bibitem{burchard} A.~Burchard,
{\em Cases of equality in the Riesz rearrangement inequality}, 
Ann. of Math. (2) 143 (1996), no. 3, 499--527

\bibitem{carlenliebloss}
E.~A.~Carlen, E.~H.~Lieb and M.~Loss, 
{\em A sharp analog of Young's inequality on $S^N$ and related entropy inequalities}, 
Jour.~Geom.~Anal.~14 (2004), 487--520



\bibitem{christradon}
M.~Christ,
{\em Extremizers of a Radon transform inequality} 
in {\em Advances in analysis: the legacy of Elias M.~Stein}, 
84--107, Princeton Math. Ser., 50, Princeton Univ. Press, Princeton, NJ, 2014

\bibitem{christyoungest} 
\bysame,
{\em Near extremizers of Young's inequality for $\reals^d$}, preprint, math.CA arXiv:1112.4875

\bibitem{christBLLeq}
\bysame, {\em Equality in Brascamp-Lieb-Luttinger inequalities}, arXiv:1706.02778 math.CA


\bibitem{christRSult} 
\bysame,
{\em A sharpened Riesz-Sobolev inequality}, 
preprint, math.CA arXiv:1706.02007 





\bibitem{christflock} 
M.~Christ and T.~C.~Flock,
{\em Cases of equality in certain multilinear inequalities of 
Hardy-Riesz-Rogers-Brascamp-Lieb-Luttinger type}, J.~Funct.~Anal.~ 267 (2014), no. 4, 998--1010

\bibitem{christoneill}
M.~Christ and K.~O'Neill, {\em Maximizers of Rogers-Brascamp-Lieb-Luttinger functionals in higher dimensions}, preprint, math.CA arXiv:1712.00109 

\bibitem{drouot} 
A.~Drouot, 
{\em Sharp constant for a $k$-plane transform inequality},
Anal.~PDE, 7 (6) (2014), pp.~1237--1252

\bibitem{flock}
T.~C.~Flock, {\em Uniqueness of extremizers for an endpoint inequality of the 
$k$-plane transform}, J.~Geom.~Anal.~26 (2016), no.~1, 570--602

\bibitem{wigdersonetal}
A.Garg, L.~Gurvits, R.~Oliveira, and A.~Wigderson, 
{\em Algorithmic and optimization aspects of Brascamp-Lieb inequalities, via operator scaling}, 
Geom.~Funct.~Anal. 28 (2018), no.~1, 100--145

\bibitem{hardyetal}
G.~E.~Hardy, J.~E.~Littlewood, and G.~P\'olya,
{\em Inequalities}, Cambridge University
Press, London and New York (1952)

\bibitem{liebgaussian} 
E.~Lieb,
{\em Gaussian kernels have only Gaussian maximizers}, 
Invent. Math. {\bf 102} (1990), 179--208.

\bibitem{liebloss}
E.~Lieb and M.~Loss,
{\em Analysis}, Amer. Math. Soc., Providence, RI, 1997

\bibitem{riesz}
F.~Riesz,
{\em Sur une in\'egalit\'e int\'egrale}, J. London Math. Soc. 5 (1930), 162--168

\bibitem{rogers1}
C.~A.~Rogers, 
{\em Two integral inequalities}, 
J. London Math. Soc. 31 (1956), 235--238

\bibitem{rogers2}
\bysame,
{\em A single integral inequality}, 
J. London Math. Soc. 32 (1957), 102--108

\bibitem{sobolev}
S.~L.~Sobolev, {\em On a theorem of functional analysis}, 
Mat. Sb. (N.S.) 4 (1938), 471--497



\end{thebibliography}
\end{document}